\documentclass{amsart}

\usepackage{latexsym}
\usepackage{amsfonts}
\usepackage{amssymb}
\usepackage{amsmath}
\usepackage{amscd}
\usepackage{amsthm}
\usepackage[T1]{fontenc}
\usepackage{delarray} 	
\usepackage{latexsym}  		
\usepackage{enumerate}		
\usepackage[hypertex]{hyperref}

\theoremstyle{plain}
\newtheorem{theo}{Theorem}[section]
\newtheorem*{theo*}{Theorem}
\newtheorem{cor}[theo]{Corollary}
\newtheorem*{cor*}{Corollary}

\newtheorem{main}{Theorem}
\newtheorem{lemma}[theo]{Lemma}
\newtheorem{prop}[theo]{Proposition}

\theoremstyle{definition}
\newtheorem{dfn}[theo]{Definition}
\newtheorem{propdef}[theo]{Proposition/Definition}
\newtheorem*{nota}{Notation}
\newtheorem{cont}[theo]{Context}
\newtheorem{rem}[theo]{Remark}
\newtheorem{ex}[theo]{Example}

\newcommand{\f}{\mathfrak}
\newcommand{\mb}{\mathbb}

\newcommand{\mf}{\mathbf}
\newcommand{\mc}{\mathcal}

\newcommand{\N}{\mathbb{N}}                 			
\newcommand{\Z}{\mathbb{Z}}                 			
\newcommand{\Q}{\mathbb{Q}}                 			
\newcommand{\R}{\mathbb{R}}                 			
                 			
\newcommand{\F}{\mathbb{F}}                 			

\renewcommand{\P}{\mathcal{P}}
\newcommand{\ox}{\otimes}

\renewcommand{\O}{\mathcal{O}}
\newcommand{\Fc}{\mathcal{F}}
\newcommand{\Cc}{\mathcal{C}}
\newcommand{\Lc}{\mathcal{L}}
\renewcommand{\Mc}{\mathcal{M}}
\newcommand{\Gc}{\mathcal{G}}
\newcommand{\Kc}{\mathcal{K}}

\newcommand{\foral}{\; \forall \;}

\newcommand{\Ra}{\Rightarrow}
\newcommand{\Lra}{\Leftrightarrow}
\newcommand{\WD}[1]{#1}
\newcommand{\cln}[1]{[ #1 ]_{\operatorname{num}}}
\newcommand{\NeSe}[1]{N^1( #1 )}

\DeclareMathOperator{\im}{im}										
							
\DeclareMathOperator{\Spec}{Spec}

\DeclareMathOperator{\ch}{char}

\DeclareMathOperator{\codim}{codim}
\DeclareMathOperator{\rk}{rk}
\DeclareMathOperator{\divd}{div}
\DeclareMathOperator{\id}{id}

\DeclareMathOperator{\vol}{vol}

\DeclareMathOperator{\Aut}{Aut}
\DeclareMathOperator{\Gal}{Gal}
\DeclareMathOperator{\Div}{Div}
\DeclareMathOperator{\DIV}{\mathbf{Div}}

\DeclareMathOperator{\Pic}{Pic}
\DeclareMathOperator{\PIC}{\mathbf{Pic}}

\DeclareMathOperator{\codepth}{codepth}
\DeclareMathOperator{\depth}{depth}

\DeclareMathOperator{\Eff}{Eff}

\DeclareMathOperator{\GIDiv}{GIDiv}
\DeclareMathOperator{\GIDIV}{\mf{GIDiv}}

\title{A Chebotarev-type density theorem for divisors on algebraic varieties}
\author{Armin Holschbach}

\usepackage[active]{srcltx}

\begin{document}

\begin{abstract}
Let $Z \to X$ be a finite branched Galois cover of normal projective geometrically integral varieties of dimension $d \geq 2$ over a perfect field $k$. For such a cover, we prove a Chebotarev-type density result describing the decomposition behaviour of geometrically integral Cartier divisors.
As an application, we classify Galois covers among all finite branched covers of a given normal geometrically integral variety $X$ over $k$ by the decomposition behaviour of points of a fixed codimension $r$ with $0 < r < \dim X$.

\end{abstract}

\maketitle

\markright{A CHEBOTAREV-TYPE DENSITY THEOREM FOR DIVISORS}
\pagestyle{headings}

\setcounter{tocdepth}{1}
\tableofcontents

\section{Introduction}

Let $Z \to X$ be a finite branched Galois cover (with Galois group $G$) of normal varieties over a field $k$, i.e.\ a finite dominant morphism together with a finite group $G \subseteq \Aut(Z)$  such that $X$ is (isomorphic to) the quotient scheme $Z/G$. The examination of the decomposition behaviour of points of $X$ in $Z$ is one of the central tools in the study of such covers.

For closed points on $X$, the decomposition behaviour depends largely on the structure of $k$. If $k$ is algebraically closed, every closed point of $X$ that is unramified in $Z$ splits completely for obvious reasons; if $k$ is a finite field, the asymptotic decomposition behaviour of closed points is described by the Chebotarev density theorem in its extended version by Serre (\cite[Theorem 7]{Ser65}).  More precisely, this theorem gives the Dirichlet density of the set of closed points whose Frobenius class lies in a given conjugacy class of $G$.

If we have $d= \dim X  \geq 2$ and take a look at the decomposition behaviour of non-closed points, the decomposition behaviour is mostly independent of $k$. In fact, it can be deduced from \cite{Sch34} that for any given integer $r$ with $0 < r < d$ and any subgroup $H$ of $G$, $H$ occurs infinitely often as decomposition group of a point on $Z$ of codimension $r$ (see Remark \ref{remschmidt}). This fact, however, does not give any information about the asymptotic decomposition behaviour of such points.

In this paper, we will describe a Chebotarev-type density theorem for the decomposition behaviour of points on $X$ of codimension $1$, i.e.\ Weil prime divisors. For a given conjugacy class $\Cc$ of subgroups of $G$, we determine the density of the set of such points $x$ that have decomposition class $\Cc$, i.e.\ such that $\Cc$ is the conjugacy class of the decomposition group of a point in $Z$ lying over $x$.

The basic idea is the following: We choose a fixed ample Cartier divisor $D_0$ on $X$ and consider the Weil prime divisors that are linearly equivalent to some multiple of $D_0$, identifying Cartier divisors with their corresponding Weil divisors.  For each positive integer $m$, the prime divisors in the linear system $|mD_0|$ (considered as a projective space) form an open subvariety by Bertini's theorem. The prime divisors with decomposition class $\Cc$ form a subvariety thereof, and the quotient of the dimensions of these two varieties (prime divisors with decomposition class $\Cc$ versus all prime divisors) converges to a limit as $m$ approaches infinity, measuring the density of prime divisors with this decomposition class.

In order to get this machinery to work, we have to specify our setting: We will consider finite branched Galois covers $Z \to X$ with Galois group $G$ of normal, projective, geometrically integral varieties over a perfect field $k$.

For any Cartier divisor $D$ on $X$ and any field extension $K|k$, let $P_D(K)$ be the set consisting of all geometrically integral effective Cartier divisors on $X_K := X \times_k K$ that are linearly equivalent to $D_K$, the flat pullback of $D$ under the canonical projection $X_K \to X$. For fixed $D$, this defines a functor $P_D$.

For a geometrically integral divisor $D'$ on $X$, we define its \emph{geometric decomposition class} in the Galois cover $Z \to X$ to be the decomposition class of $D'_{\bar k}$ in $Z_{\bar k} \to X_{\bar k}$ , where $\bar k$ is an algebraic closure of $k$. Then for a given Cartier divisor $D$ on $X$ and a conjugacy class $\Cc$ of subgroups of $G$, we define a functor $D^\Cc_D$ from the category of field extensions of $k$ to the category of sets by
\[D^\Cc_D(K) = \biggl\{ D' \in P_D(K) \biggl| \begin{matrix}  D' \text{ is unramified  in } Z_K \to X_K  \text{ and } \\  \text{ has geometric decomposition class } \Cc \end{matrix} \biggr\}.\]

Now we can state our result precisely:

\begin{main}[\textbf{\ref{dens1}}]\label{main1}
Let $Z \to X$ be a finite branched Galois cover with Galois group $G$ of normal, geometrically integral projective varieties over a perfect field $k$, and assume that $\dim X = d\geq 2$ and $Z$ has a $G$-equivariant resolution of singularities. Let $\Cc$ be a conjugacy class of subgroups of $G$. Then for any Cartier divisor $D$ on $X$, there exist a quasiprojective variety $\mf P_D$ and a reduced subscheme $\mf D^\Cc_D$ of $\mf P_D$ representing the functors $P_D$ and $D^\Cc_D$ defined above.

If $D_0$ is an ample Cartier divisor that is (linearly equivalent to) the push-forward of an effective Cartier divisor on $Z$, then
\[ \lim_{m \to \infty} \frac{\dim \mf{D}^\Cc_{mD_0}}{\dim \mf{P}_{mD_0}} = \frac{1}{(G:\Cc)^{d-1}},\]
where $(G : \Cc)$ is defined to be $(G :H)$ for any representative $H$ of $\Cc$. If we just assume $D_0$ to be ample, the statement still holds if we replace $D_0$ by $\# G \cdot D_0$ or regard the limit superior instead of the limit.
\end{main}

Note that the condition that $Z$ has a $G$-equivariant resolution of singularities is automatically fulfilled if $d=2$ or $\ch k = 0$, by Lipman's theorem (\cite{Lip78}) and the equivariant version of Hironaka's theorem (e.g., \cite{AbW97}), respectively.

If we want to extend our result to describe the decomposition behaviour of all prime divisors, we can sort the divisors by degree: Fix an ample Cartier divisor $D_0$ on $X$, and for any Cartier divisor $D$ on $X$, call the intersection number $D_0^{d-1} \cdot D$ the degree of $D$. The corresponding density theorem for this sorting is Theorem \ref{revdens}.

Returning to the case of a finite field $k$, we can actually \emph{count} divisors by their degree, considering prime divisors and their decomposition classes instead of their geometric counterparts. Using Poonen's Bertini theorems over finite fields (\cite{Poo04}) instead of the original Bertini theorem, we can modify our methods to prove the following result:

\begin{main}[\textbf{\ref{dens3}}]\label{main2}
Assume that $k$ is a finite field, and that $g: Z \to X$ and $\Cc$ satisfy the assumptions of Theorem \ref{main1}. For any $r \in \R$, let $p_\#(r)$ be the number of Cartier prime divisors of degree at most $r$, and let $d^\mc{C}_\#(r)$ be the number of those divisors that additionally are unramified in the Galois cover $g:Z\to X$ and have decomposition class $\Cc$. Then
\[\lim_{r\to\infty}\frac{\log d^\mc{C}_\#(r)}{\log p_\#(r)} = \frac{1}{(G:\mc{C})^{d-1}}.\]
\end{main}

Neither one of the `densities' used in the theorems above (meaning the expressions on the left hand side of the equations) behaves in the way one would normally expect densities to behave; for example, they do not act additively on a union of disjoint sets. Still, we can derive some analogues of corollaries of the original Chebotarev density theorem. To give an example, let us state a theorem that can be regarded as an analogue of a theorem of M. Bauer (\cite[Theorem 13.9]{Neu99}), which identifies Galois extensions of number fields by the splitting behaviour of their primes.

For a finite branched cover of normal varieties $f: Y \to X$, we say that a point $x \in X$ \emph{splits in $Y$} if it is unramified in $Y$ and there exists a point $y \in Y$ lying over $x$ such that the residue fields at $y$ and $x$ coincide. Then the following holds:

\begin{main}[\textbf{\ref{bauer}}]\label{main3}
Let $X, Y, Z$ be normal geometrically integral quasiprojective varieties of dimension $d \geq 2$ over a field $k$ of characteristic zero; let $f: Y \to X$ be a finite branched cover, $g : Z \to X$ be a finite branched Galois cover. Fix an positive integer $r < d$. Then the following are equivalent:
\begin{enumerate}[(a)]
\item $f: Y \to X$ factors through $g: Z \to X$.
\item Every point $x \in X$ of codimension $r$ that splits in $Y$ splits in $Z$.
\end{enumerate}
\end{main}

A birational version of Theorem \ref{main3} in the essential case $r=1$ was already proven by F.K.\ Schmidt in \cite{Sch34}.

\vspace{.3cm}

\paragraph{\textbf {Structure of this paper.}} In section \ref{secpfpb}, we review the classical push-forward and pull-back maps of Cartier divisors and construct scheme-theoretic pull-back and push-forward maps for the scheme representing relative effective divisors. Section \ref{secmt} contains the statement of the main theorem and the construction of $\mf P_D$ and $\mf D^\Cc_D$.
In order to determine the asymptotic behaviour of $\dim \mf P_{mD_0}$ and $\dim \mf D^\Cc_{mD_0}$, the notion of the volume of a divisor is needed. In section \ref{secvol}, we recall its definition and consider its behaviour under push-forward and pull-back of divisors. This theory is applied in section \ref{secasy} to finish the proof of the main theorem. The remaining section \ref{secapp} contains modifications and applications of Theorem \ref{main1}, among them Theorems \ref{main2} and \ref{main3}.

\begin{nota}
A variety over a field $k$ will be an integral scheme of finite type over $k$, unless noted otherwise. A finite branched cover of two varieties is a finite dominant generically \'{e}tale morphism.

The function field of an integral scheme $X$ is denoted by $K(X)$. If $x$ is a point on a scheme $X$, then $\kappa(x)$ denotes its residue field.

For two Cartier divisors $D, D'$ on a projective variety $X$, we write $D \equiv D'$ if the two divisors are numerically equivalent. The numerical equivalence class of a divisor $D$ is denoted by $\cln{D}$, the one of a line bundle $\Lc$ is denoted by $\cln{\Lc}$.

The group of Cartier divisors and the group of line bundles on a variety $X$ will be denoted by $\Div(X)$ and $\Pic(X)$, respectively. The N\'{e}ron-Severi group of Cartier divisors on $X$ modulo numerical equivalence is denoted by $\NeSe{X}$; by the N\'{e}ron-Severi theorem (\cite[XIII.5.1]{BGI71}), it is a finitely generated free $\Z$-module. We set $\NeSe{X}_\R := \NeSe{X}\otimes_\Z \R$ and view $\NeSe{X}$ as a lattice in it. Inside $\NeSe{X}_\R$, we define the pseudoeffective cone $\overline\Eff(X)$ to be the closed convex cone generated by the classes of effective Cartier divisors.

For Cartier divisors $D_1, \ldots, D_d$ on a projective variety $X$ of dimension $d$, we denote their intersection number (cf.\ \cite[Definition 2.4.2]{Ful98}) by $D_1 \cdots D_d$. Since this number is invariant under numerical equivalence, we can extend this notation to $\NeSe{X}$ and $\NeSe{X}_\R$.

For a Cartier divisor $D$ on a projective variety $X$ over $k$, we denote the dimension of $H^0(X,\O_X(D))$ over $k$ by $h^0(X,D)$; similarly, for an invertible sheaf $\Lc$ on $X$, we set $h^0(X,\Lc) := \dim H^0(X,\Lc)$.

For $r \in\R$, we set $\lfloor r \rfloor$ and $\lceil r \rceil$ to be the greatest integer less or equal to  $r$ and the smallest integer greater or equal $r$, respectively.
\end{nota}

\paragraph{\textbf{Acknowledgements.}} This paper is mostly adapted from my PhD thesis, and I would like to thank my advisor, Florian Pop, for his constant encouragement and guidance. Thanks also to Alexander Schmidt for many valuable comments and suggestions, and to Ted Chinburg for bringing up a question that led to Theorem \ref{dens2}.

\section{Push-forward and Pull-back}\label{secpfpb}

Though push-forward and pull-back maps on divisors are a very common tool in algebraic geometry, they are most often used in the context of flat morphisms. What we need in the following sections is the notion of push-forward and pull-back for divisors and relative effective divisors in the case of a finite branched cover of normal varieties. For lack of suitable references, we present the needed facts on those push-forward and pull-back maps in the following section. The reader familiar with the standard facts from \cite[{\S}{\S} 1 \& 2]{Ful98} or \cite[{\S} 21]{Gro67} and the notion of relative effective divisors is invited to move forward to Section \ref{secmt} after having read Lemmata \ref{schpb} and \ref{schpf}.

\subsection{Basic properties}
We first recollect some basic properties of pull-back and push-forward maps for divisors, referring to \cite[21.4 \& 21.5]{Gro67} for proofs. Throughout this subsection, let $X, Y$  be normal varieties over a field $k$.

A generically finite dominant morphism $f: Y \to X$ induces a pull-back homomorphism on Cartier divisors $f^*: \Div(X) \to \Div(Y)$ which locally comes from the inclusion $K(X)^\times \hookrightarrow K(Y)^\times$. This homomorphism preserves linear equivalence and thus induces a homomorphism $f^*: \Pic(X) \to \Pic(Y)$, which is nothing else but the usual sheaf-theoretic pull-back map on line bundles.

 A finite dominant morphism $f: Y \to X$ induces a push-forward homomorphism on Cartier divisors $f_*: \Div(Y) \to \Div(X)$ which locally comes from the norm map $N_{K(Y)|K(X)}: K(Y)^\times \to K(X)^\times$.
This homomorphism also preserves linear equivalence and thus induces a homomorphism $\Pic(Y) \to \Pic(X)$.

If we identify Cartier divisors with their corresponding Weil divisors, then both $f^*: \Div(X) \to \Div(Y)$ and $f_*: \Div(Y) \to \Div(X)$ can be regarded as restrictions of analogous maps on Weil divisors. For example, for a Weil prime divisor $W$ on $Y$ one defines its push-forward by $f_*(W) = [K(W):K(f(W))] f(W)$.

In the case when $f: Y \to X$ is a finite branched cover, both the pull-back and the push-forward map exist, so we can talk about their compositions. The composition map $f_* \circ f^*: \Div(X) \to \Div(X)$ is just multiplication by $\deg(f)$. The other composition $f^* \circ f_*: \Div(Y) \to \Div(Y)$ is somewhat more complicated: If  $f: Y \to X$ is a Galois cover with Galois group $G$, then $(f^* \circ f_*) (E) = \sum_{\sigma \in G} \sigma E$ for any $E \in \Div(Y)$. This is a special case of the following

\begin{lemma}\label{pfthrupb}
Let $Z$ be a normal variety over a field $k$ and let $H \subseteq G$ be finite subgroups of $\Aut_k(Z)$ such that the quotient variety $X = Z/G$ exists. Set $Y = Z/H$ and denote the canonical morphisms $Y \to X$, $Z \to X$, and $Z\to Y$ by $f, g$ and $h$, respectively. Then, for any $E \in \Div(Y)$, $D \in \Div(X)$,
\[f_* E = D \quad \Lra \quad \sum_{\bar\sigma \in G/H} \sigma h^* E = g^* D.\]
\end{lemma}

\begin{proof}
Let us first note that $\sigma h^* E$ is well-defined since $h^* E$ is $H$-invariant. Set $F = \sum_{\bar\sigma \in G/H} \sigma h^* E$. Then both $F$ and $g^* D$ are $G$-invariant. Two $G$-invariant divisors on $Z$ coincide if their push-forwards are the same. Since $g_* g^* D = \# G \cdot D$ and $g_* F = (G:H) g_* h^* E = (G:H) f_* (h_* h^* E) = \# G \cdot f_* E$, the assertion follows.
\end{proof}

\begin{rem}\label{makeGal}
Let $f: Y \to X$ be a finite branched cover of normal varieties over a field $k$. Then by Lemma \ref{pfthrupb},  the push-forward map on divisors can be described by using pull-back maps. Indeed, in order to get to the situation of the lemma, we can take a finite Galois extension $L|K(X)$ containing $K(Y)$ and construct the normalization of $X$ in $L$. We arrive at a normal variety $Z$ with function field $K(Z)=L$ and groups $G = \Gal(L|K(X))$ and $H = \Gal(L|K(Y)) \subseteq G$ such that $G$ acts on $Z$, and $Z/G \simeq X$, $Z/H \simeq Y$. A special case arises if we take $L$ to be a Galois closure of $K(Y)|K(X)$; we then call $g:Z \to X$ the \emph{Galois closure} of $Y\!\to\! X$.
\end{rem}

\subsection{Relative effective divisors and the Picard scheme}

We want to extend the definitions of push-forward and pull-back into a more functorial setting. In order to do this, we introduce the scheme $\DIV_{X/k}$ representing relative effective divisors and the Picard scheme $\PIC_{X/k}$ and list some basic properties. For proofs, we refer to the article \cite{Kle05}.

In the following, let $X$ be a normal geometrically integral projective variety $X$ over a perfect field $k$; such a variety is automatically geometrically normal as well by \cite[Proposition 6.7.7]{Gro65}.

There exists a scheme $\DIV_{X/k}$ representing the functor
\[ \Div_{X/k}(T) := \{ \text{relative effective Cartier divisors on }X_T/T\},\]
where for an arbitrary $k$-scheme $T$, a relative effective divisor on $X_T /T$ is an effective divisor on $X_T := X \times_k T$ that is $T$-flat. Equivalently, a relative effective divisor on $X_T/T$ is a closed subscheme $D \subset X_T$ such that for every $x \in D$, $D$ is cut out at $x$ by one element that is regular on the fiber $X_t$, where $t \in T$ is the image of $x$ under the canonical projection $X_T \to T$.

We also have a group scheme $\PIC_{X/k}$ (called the Picard scheme) representing the sheaf associated to the relative Picard functor \[\Pic_{X/k}(T) := \Pic(X_T)/\Pic(T)\] in the \'{e}tale topology. There exists a proper morphism (called the \emph{Abel map}) \[\mf{A}_{X/k}: \DIV_{X/k} \to \PIC_{X/k}\] which represents the functor sending a relative effective divisor $D$ on $X_T/T$ to the sheaf $\O_{X_T}(D)$. Given a Cartier divisor $D$ on $X$, the fiber of $\mf{A}_{X/k}$ over the point $\O_X(D) \in \PIC_{X/k}$ is exactly the projective space $\mf L_D$ associated with the linear system $|D|$.

The subset of $\PIC_{X/k}$ consisting of points corresponding to numerically trivial invertible sheaves forms an open and closed subgroup scheme $\PIC^\tau_{X/k}$ of $\PIC_{X/k}$. For a given Cartier divisor $D$ on $X$, let $\DIV_{X/k}^{\cln{D}}$ denote the preimage of \[\PIC_{X/k}^{\cln{D}} := \O_X(D) + \PIC^\tau_{X/k}\] under the Abel map. Both  $\DIV_{X/k}^{\cln{D}}$ and $\PIC_{X/k}^{\cln{D}}$ are proper over $k$: the latter one is just a translate of the projective scheme $\PIC^\tau_{X/k}$, and the statement for the first one follows since the Abel map is proper.

\subsection{Scheme-theoretic Pull-back and Push-forward}

Our next aim is to extend the notion of pull-back and push-forward to relative effective divisors.
Throughout this subsection, let $f: Y \to X$ be a morphism between normal geometrically integral projective varieties over a perfect field $k$. In the case where $f$ is finite and flat, push-forward and pull-back maps between $\Div_{Y/k}$ and $\Div_{X/k}$ are constructed in \cite[21.15]{Gro67}. We slightly adapt these constructions to our needs.\\

Assume $f: Y \to X$ is a generically finite dominant morphism. The sheaf-theoretic pull-back map induces a natural transformation $f^*: \Pic_{X/k} \to \Pic_{Y/k}$, which is represented by a morphism of Picard schemes $\mf f^*: \PIC_{X/k} \to \PIC_{Y/k}$. We can also define pull-backs of relative effective divisors: For a relative effective divisor $D$ on $X_T/T$, we set its pull-back under $f_T: Y_T \to X_T$ to be $f_T^* D := D \times_{X_T} Y_T$, identifying the divisors and the corresponding subschemes of $X_T$ and $Y_T$. This is indeed a relative effective divisor on $Y_T/T$: Let $y \in Y_T$ be any point, and set $x = f_T(y)$. Then if $D$ is cut out by  $a$ in a neighborhood of $x$, $f_T^* D$ is cut out by $f_T^\#(a)$ in a neighborhood of $y$; because $a$ is regular on $X_t$, where $t$ is the image of $x$ under the canonical projection, $f_T^\#(a)$ is regular on $Y_t$, since $X_t$ and $Y_t$ are integral and $Y_t \to X_t$ is dominant. This map behaves functorially and preserves linear equivalence, its induced map on invertible sheaves is the sheaf-theoretic pull-back. Therefore, we have proved:

\begin{lemma}\label{schpb}
For a generically finite dominant morphism $f: Y \to X$, there exist pull-back morphisms $\mf{f}^*: \DIV_{X/k} \to \DIV_{Y/k}$ and $\mf{f}^*: \PIC_{X/k} \to \PIC_{Y/k}$ representing the usual pull-back of divisors and invertible sheaves. These morphisms fit into a commutative diagram

\[\begin{CD}
 \DIV_{X/k}  @>\mathbf{f}^*>>\DIV_{Y/k}\\
@V {\mf A}_{X/k}VV @VV{\mf A}_{Y/k}V\\
\PIC_{X/k} @>\mathbf{f}^*>> \PIC_{Y/k}.
\end{CD}\]
\end{lemma}

We want to construct push-forward maps in a similar fashion. So assume that $f: Y \to X$ is a finite dominant morphism.

Let us first consider the local situation:
Let $B|A$ be an integral extension of normal domains over a field $k$ such that the corresponding extension of quotient fields $L|K$ is finite. Let $R$ be any $k$-algebra. Then $L_R := L \ox_k R$ is a finite free module over $K_R := K \ox_k R$. We consider the norm map
\[N_{L_R | K_R}:L_R \to K_R, \quad \lambda \mapsto \det (m_\lambda),\]
where $m_\lambda$ is the endomorphism of the free $K_R$-module $L_R$ given by multiplication with $\lambda$ and $\det (m_\lambda)$ is its determinant (cf.\ \cite[III.9.1, Def.\ 1]{Bou07a}).

\begin{lemma}\label{norm}
The restriction of $N_{L_R | K_R}$ to the subring $B_R := B \ox_k R$ of $L_R$ maps to $A_R := A \ox_k R \subseteq K_R$; we call it $N_{B_R | A_R}$.
\end{lemma}

\begin{proof}
If $\varphi: S \to R$ is a homomorphism of $k$-algebras, then we have a commutative diagram
\[\begin{CD}
L_S @>\id\otimes\varphi>> L_R\\
@VN_{L_S | K_S}VV @VVN_{L_R | K_R}V\\
K_S @>\id\otimes\varphi>> K_R
\end{CD}\]
by \cite[III.9.1.(12)]{Bou07a}. As any element of $A_R$ lies in $(\id\otimes\varphi)(A_S)$ for a suitable polynomial algebra $S$ over $k$ and a suitable homomorphism $\varphi: S \to R$, we can assume that $R=k[X_1, \ldots, X_n]$. In this case, $A_R = A[X_1, \ldots, X_n]$ is integrally closed (\cite[V.3, Cor.\ 2 of Prop.\ 13]{Bou06a}), and since $B_R$ is integral over $A_R$, the lemma holds by \cite[V.6, Cor.\ 2 of Prop.\ 17]{Bou06a}.
\end{proof}

Coming back to our situation, we define a map
\[N_{{f_T}_* \O_{Y_T}|\O_{X_T}}: f_* \O_{Y_T} \to \O_{X_T}\]
by virtue of the preceding lemma (gluing the local data). It is clear from the definition that this morphism is multiplicative, hence gives a morphism of sheaves of abelian groups $N_{{f_T}_* \O_{Y_T}|\O_{X_T}}: f_* \O_{Y_T}^\times \to \O_{X_T}^\times$. This morphism induces a norm homomorphism
\[\Pic(Y_T) \simeq H^1(Y_T, \O_{Y_T}^\times) \to H^1(X_T, \O_{X_T}^\times) \simeq \Pic(X_T)\]
and, by functoriality, a morphism $\mf f_*: \PIC_{Y/k} \to \PIC_{X/k}$.

We can also define the push-forward of an relative effective divisor $D$ on $Y_T/T$. $(f_T)_* D$ is best defined locally in the following way: For each open subset $V$ of $Y_T$ of the form $V=(f_T)^{-1}(U)$ with $U=\Spec A$ affine, $D$ is cut out on $V$ by an element $b \in B = \O_V(V)$. We define $(f_T)_* D$ on $U$ to be cut out by $N_{B|A}(b) \in A$. Using Lemma \ref{norm}, we can see that these local descriptions fit together to give a closed subscheme of $X_T$; again, $a$ is regular on every $X_t$ with $t$ in the image of $U$ under the canonical projection $X_T \to T$ since $Y_t \to X_t$ is a dominant morphism of integral schemes. Therefore, $(f_T)_* D$ is indeed a relative effective divisor on $X_T$. This push-forward map on relative effective divisors behaves functorially and preserves linear equivalence; the induced map on invertible sheaves is just the map defined above. We have thus proved the following

\begin{lemma}\label{schpf}
 A finite dominant morphism $f:Y \to X$ induces push-forward morphisms $\mathbf{f}_*: \DIV_{Y/k} \to \DIV_{X/k}$ and $\mathbf{f}_*: \PIC_{Y/k} \to \PIC_{X/k}$ representing the push-forward of divisors and invertible sheaves. The following diagram commutes:
\[\begin{CD}
 \DIV_{Y/k}  @>\mathbf{f}_*>>\DIV_{X/k}\\
@V {\mathbf A}_{Y/k}VV @VV{\mathbf A}_{X/k}V\\
\PIC_{Y/k} @>\mathbf{f}_*>> \PIC_{X/k}.
\end{CD}\]
\end{lemma}

\section{The Main Theorem}\label{secmt}

We turn our attention to the main theorem. Let us shortly review the context:
Let $g: Z \to X$ be a finite branched Galois cover with Galois group $G$ of normal, geometrically integral projective varieties over a perfect field $k$, and assume that $\dim X = d\geq 2$ and $Z$ has a $G$-equivariant resolution of singularities. We want to classify those prime divisors on $X$ with  decomposition class $\Cc$, where $\Cc$ is a given conjugacy class of subgroups of $G$.

We remind the reader of the two functors defined in the introduction: For any Cartier divisor $D$ on $X$ and any field extension $K|k$, the set $P_D(K)$ consists of all geometrically integral effective Cartier divisors on $X_K := X \times_k K$ that are linearly equivalent to $D_K$, and $D^\Cc_D(K)$ is the set of those divisors in $P_D(K)$ that additionally are unramified in the cover $Z \to X$ and have geometric decomposition class $\Cc$.

In this situation, the following holds:

\begin{theo}\label{dens1}
For any Cartier divisor $D$ on $X$, there exist a quasiprojective variety $\mf P_D$ and a reduced subscheme $\mf D^\Cc_D$ of $\mf P_D$ representing the functors $P_D$ and $D^\Cc_D$, respectively.

Let $D_0$ be an ample Cartier divisor that is (linearly equivalent to) the push-forward of an effective Cartier divisor on $Z$. Then
\[ \lim_{m \to \infty} \frac{\dim \mf{D}^\Cc_{mD_0}}{\dim \mf{P}_{mD_0}} = \frac{1}{(G:\Cc)^{d-1}},\]
where $(G : \Cc)$ is defined to be $(G :H)$ for any representative $H$ of $\Cc$. If we just assume $D_0$ to be ample, the statement still holds if we replace $D_0$ by $\# G \cdot D_0$ or regard the limit superior instead of the limit.
\end{theo}

In this section, we show the existence of $\mf P_D$ and $\mf D^\Cc_D$. The asymptotic behaviour will be described in Section \ref{secasy}.

\subsection{Existence of $\mf P_{D}$}\label{secexP}
Let us assume that $X$ is a geometrically integral projective variety over $k$. We will construct $\mf P_D$ by intersecting the scheme $\mf L_D$ representing (all divisors in) the linear system $|D|$ with the following subscheme of $\DIV_{X/k}$:

\begin{propdef}
There exists a unique open subscheme $\GIDIV_{X/k}$ of $\DIV_{X/k}$ representing the functor $\GIDiv_{X/k}$ given by
\[\GIDiv_{X/k}(T) = \{D \in \Div_{X/k}(T) \ | \ D_t \text{ is geometrically integral} \ \forall \ t \in T\}.\]
\end{propdef}

\begin{proof}
To a point $z \in \DIV_{X/k}$, we can associate an effective Cartier divisor $D(z)$ on $X_{\kappa(z)}$ which corresponds to the natural morphism $\Spec \kappa(z) \to \DIV_{X/k}$ with image $\{z\}$. We consider the set
\[ U = \{ z \in \DIV_{X/k} \ |\ D(z) \text{ is geometrically integral} \}. \]

Let $D$ be a relative effective divisor on $X_T/T$, $\varphi: T \to \DIV_{X/k}$ be the corresponding morphism.

\noindent\emph{Claim:} $\varphi^{-1}(U) = \{ t \in T \ |\ D_t$ is geometrically integral $\}$.

In fact, $D_t$ corresponds to the morphism $\Spec \kappa(t) \to \DIV_{X/k}$ given by the composition of the natural morphism $\Spec \kappa(t) \to T$ and $\varphi$, so by the definition of $U$, $D_t$ is geometrically integral if and only if $\varphi(t) \in U$.

$D$ is proper and flat over $T$, so by \cite[12.2.4 (vii)]{Gro66} the set of all $t \in T$ for which $D_t$ is geometrically integral is open in $T$, i.e.\ $\varphi^{-1}(U)$ is an open subset of $T$. For $T = \DIV_{X/k}$, $\varphi = \id_{\DIV_{X/k}}$, this shows that $U$ is open.

Let $\GIDIV_{X/k}$ be the open subscheme of $\DIV_{X/k}$ with underlying set $U$. The claim implies that $D \in \GIDiv_{X/k}(T)$ if and only if $\varphi: T \to \DIV_{X/k}$ factors through $\GIDIV_{X/k}$, which proves the assertion.
\end{proof}

\begin{rem}
If $K$ is an algebraically closed field extension of $k$, then $\GIDiv_{X/k}(K)$ is just the set of Cartier divisors on $X_K$ that correspond to a Weil prime divisor, or in other words, the set of locally principal Weil prime divisors on $X_K$.
\end{rem}

\begin{dfn}
For a given Cartier divisor $D$ on $X$, we define $\mf{P}_{D}$ to be the following scheme-theoretic intersection (inside $\DIV_{X/k}$):
\[\mf{P}_{D} = \mf{L}_{D} \cap \GIDIV_{X/k} := \mf{L}_{D} \times_{\,\DIV_{X/k}}  \GIDIV_{X/k}.\]
$\mf{P}_D$ represents the geometrically integral divisors that are linearly equivalent to $D$.
\end{dfn}

\begin{rem}\label{unr}
In the situation of our main theorem, we are also interested in prime divisors which are unramified in a cover $Z \to X$. Since only finitely many prime divisors can be ramified, we can easily construct an open subvariety $\mf P'_D$ of $\mf P_D$ representing the unramified geometrically integral divisors in $|D|$ by removing the finitely many closed points corresponding to ramified divisors. For all but finitely many linear systems $|D|$, we have $\mf P'_D = \mf P_D$.
\end{rem}

\subsection{Behaviour of divisors in finite and generically finite covers}\label{secHDT}

Before constructing $\mf{D}_D^\Cc$, we have to find ways to classify the decomposition behaviour of divisors in covers of varieties.

\begin{cont}\label{contGal}
In this subsection, we will consider the following situation: Let $Z$ be a normal integral scheme with a finite group $G$ acting on it such that the quotient $X = Z/G$ exists. Let $H \subseteq G$ be a subgroup, and let $Y = Z/H$ be the corresponding quotient. $X$ and $Y$ are both normal integral schemes; let $f, g, h$ denote the canonical morphisms $Y \to X, Z \to X, Z \to Y$.
\end{cont}

We want to find a criterion for a prime divisor on $X$ to split in $Y$.
If we use the notion of Weil divisors, then there is an easy criterion:

\begin{lemma}\label{splitWeil}
A Weil prime divisor on $X$ splits in $Y$ if and only if it is unramified in $Y$ and the push-forward of a Weil prime divisor on $Y$.
\end{lemma}

This follows directly from the definition of push-forward for Weil divisors (see \cite[Definition 7.2.17]{Liu02}). Unfortunately, we will need to give criteria for splitting using only Cartier divisors. And even if we can assume $Z$ and $X$ to be regular, so that Weil and Cartier divisors coincide for them, the same need not to be true for $Y$:

\begin{ex}
Let $Z = \Spec C$ with $C = k[X,Y], \ch k \neq 2$. Let the Klein four group $G = \langle \sigma, \tau \rangle$ act on $Z$ by \[\sigma: X \mapsto -X, Y \mapsto -Y \quad \text{and} \quad \tau: X \mapsto Y, Y \mapsto X,\] and set $H = \langle \sigma \rangle$. Then $Z/H = Y = \Spec B$ with \[B = k[X,Y]^H = k[X^2, XY, Y^2] = k[R,S,T]/(RT-S^2),\] and furthermore $Z/G = X = \Spec A$ with \[A = k[X,Y]^G = k[X^2+Y^2, XY] = k[U,V].\] Both $X$ and $Z$ are regular, but $Y$ is not (\cite[Example II.6.5.2]{Har77}). Even more, if we consider the Cartier $D$ divisor on $X$ given by $V=0$, then $D$ splits in $Y$, but no Weil prime divisor on $Y$ lying over $D$ is locally principal: They are given by $R=S=0$ and $T=S=0$, respectively.
\end{ex}

Therefore, if we want to describe splitting in terms of Cartier divisors, we have to take a slightly more complicated approach. If $Z$ is regular, then we still get a rather easy description of split divisors:

\begin{prop}\label{splitreg}
In the Context \ref{contGal}, assume $Z$ is regular. Let $D$ be a Cartier prime divisor on $X$ that is unramified in $Z$. The following are equivalent:
\begin{enumerate}[(a)]
\item $D$ splits in $Y$.
\item There exists a Cartier prime divisor $F_{1}$ on $Z$ over $D$ such that $G_{F_{1}} \subseteq H$.
\item There is an effective $H$-stable Cartier divisor $F$ on $Z$ such that $g_* F = (\#H) D$.
\item There exists an effective $H$-stable Cartier divisor $F$ on $Z$ such that \[\sum_{\bar\sigma\in G/H} \sigma F = g^{*} D.\]
\end{enumerate}
\end{prop}

\begin{proof}
$(a) \Lra (b)$: Let $x$ be the point on $X$ corresponding to the prime divisor $D$. Take any Cartier prime divisor $F_1$ on $Z$ over $D$, let $z$ denote the corresponding point on $Z$ and $y$ the image of $z$ in $Y$. Then
\[\begin{split}
G_{F_1} \subseteq  H &\Lra G_{F_1} = H_{F_1} \\& \Lra [\kappa(z):\kappa(x)] = \# G_{F_1} = \# H_{F_1} = [\kappa(z):\kappa(y)]\\& \Lra [\kappa(y):\kappa(x)]=1,
\end{split}\]
which proves the assertion.

$(b) \Lra (c)$
Let $F_1$ be a divisor as described in $b)$, and set $F= \sum_{\bar \tau \in H / G_{F_1}} \tau F_1$. Then $F$ is $H$-stable by construction. Furthermore,
\[g_* F  = g_*\left(\sum_{\bar\tau\in H/G_{F_1}} \tau F_1\right) = (H:G_{F_1}) g_* (F_1) = (\# H) D.\]

For the inverse direction, assume we are given a divisor $F$ as in c). Let $F_1 \leq F$ be a Cartier prime divisor. Since $F$ is $H$-stable, we must have $\sum_{\bar\tau \in H/H_{F_1}} \tau F_1 \leq F$, hence
\[(\# H) D = g_* F \geq g_* \left(\sum_{\bar\tau \in H/H_{F_1}} \tau F_1\right) = (H:H_{F_1}) (\# G_{F_1}) D.\]
Thus, $\# G_{F_1} \leq \# H_{F_1} = \# (G_{F_1} \cap H)$, which implies $G_{F_1} \subseteq H$.

$(c) \Lra (d)$
Let $F$ be an effective $H$-stable Cartier divisor on $Z$. Then $\sum_{\bar\sigma\in G/H} \sigma F$ is $G$-stable by construction, and
\[g_* F = (\# H) D \Lra g_* \left(\sum_{\bar\sigma\in G/H} \sigma F\right) = (\# G) D = g_* (g^* D) \Lra \sum_{\bar\sigma\in G/H} \sigma F = g^* D.\]
The last implication follows by \cite[Example 1.7.6]{Ful98}.
\end{proof}

If $Z$ is not regular any more, but has a $G$-equivariant resolution of singularities, we still find a way to describe splitting divisors:

\begin{prop}\label{splitnonreg}
In the Context \ref{contGal}, assume that $Z$ allows a $G$-equivariant resolution of singularities $\pi: Z' \to Z$. Denote the composition $g \circ  \pi: Z' \to X$ by $g'$, and consider a Cartier prime divisor $D$ on $X$ that is unramified in $Z$. Then the following are equivalent:
\begin{enumerate}[(a)]
\item $D$ splits in $Y$.
\item There exists an effective $H$-stable Cartier divisor $F'$ on $Z'$ such that ${g'}^{*} D -\sum_{\bar\sigma\in G/H} \sigma F'$ is effective and supported in the exceptional locus of $\pi$.
\item Let $E'_1, \ldots E'_r$ be the Cartier prime divisors of $Z'$ with support in the exceptional locus of $\pi$. Then there exists an effective $H$-stable Cartier divisor $F''$ on $Z'$ and integers $d_1, \ldots, d_r$ with $0 \leq d_i < (G:H)$ for all $i$ such that
    \[ \sum_{\bar\sigma \in G/H} \sigma F'' = {g'}^* D + \sum_{i=1}^r d_i E'_i.\]
\end{enumerate}
\end{prop}

\begin{proof}
Let $i: Z_0 \hookrightarrow Z$ be the maximal open subscheme of $Z$ over which $\pi$ is an isomorphism. By definition of $Z_0$, we can also define a canonical open immersion $i': Z_0 \hookrightarrow Z'$ such that $i = \pi \circ i'$; by abuse of notation, we will consider $Z_0$ as open subscheme of both $Z$ and $Z'$. Since $\pi$ is $G$-equivariant, $Z_0$ is $G$-invariant; so we can define $X_0 = Z_0/G$ and $Y_0 = Z/G$ as open subschemes of $X$ and $Y$, respectively. Since $Z$ is normal, $\codim(Z-Z_0,Z) \geq 2$ by \cite[Corollary 4.4.3]{Liu02}, hence also $\codim(X-X_0,X) \geq 2$. In particular, if $D$ is a Cartier prime divisor of $X$, its corresponding point lies in $X_0$.

Because being split is a local criterion, $D$ splits in $Y$ if and only if $D_0 = D|_{X_0}$ splits in $Y_0$. By Proposition \ref{splitreg}, the latter is equivalent to saying that there exists an effective $H$-stable Cartier divisor on $Z_0$ such that $\sum_{\bar\sigma\in G/H} \sigma F_0 = g_0^{*} D_0$, where $g_0=g|_{Z_0} : Z_0 \to X_0$. Since $g_0^* D_0 = (g^* D)|_{Z_0}  = {i'}^* \pi^* g^* D = ({g'}^* D)|_{Z_0}$, a) is equivalent to the following condition:\\
$(a')$ There exists an effective $H$-stable Cartier divisor $F'_0$ on $Z_0$ such that \[\sum_{\bar\sigma\in G/H} \sigma F'_0 = ({g'}^{*} D)|_{Z_0}.\]
$(a') \Ra (b)$: Extend $F'_0$ to an effective Cartier divisor $F'$ on $Z'$ by considering the closure of the corresponding Weil divisor as a Weil divisor on $Z'$. Then $F'$ automatically fulfills the conditions mentioned in $(b)$.\\
$(b) \Ra (c)$: Write ${g'}^* D = \sum_{\bar\sigma \in G/H} \sigma F'+\sum_{j=1}^r a_i E'_i$. Then $F'' = F' + \sum_{j=1}^r \lceil \frac{a_i}{(G:H)} \rceil E'_i$ is $H$-stable, since both $F'$ and ${g'}^* D$ are $H$-stable, and we get the assertion.\\
$(c) \Ra (a')$: $F'_0 = F''|_{Z_0}$ does the job.
\end{proof}

\subsection{Existence of $\mf D^\mc{C}_{D}$}\label{secexD}

Let us get back to the situation at the beginning of this section: Let $Z$ be a normal geometrically integral projective variety over a perfect field $k$, $G$ a finite group of $k$-automorphisms of $Z$. Assume that $Z$ possesses a $G$-equivariant resolution of singularities $\pi: Z' \to Z$. Let $H$ be a subgroup of $G$, $\Cc$ its conjugacy class, and set $X=Z/G$ and $Y=Z/H$ as above. In order to prove the existence of $\mf D^\mc{C}_{D}$ for a fixed divisor $D$ on $X$, we will first introduce another property of divisors:

\begin{dfn}
A geometrically integral Cartier divisor $D$ on $X$ is said to \emph{split geometrically} in $Y$ if for some field extension $k'|k$, the base changed Cartier divisor $D_{k'}$ on $X_{k'} := X \times_k k'$ splits in $Y_{k'} := Y \times_k k'$.
\end{dfn}

\begin{propdef}\label{existS}
 For every Cartier divisor $D$ on $X$, there exists a closed reduced subscheme $\mf S^Y_{D}$ of $\mf P'_{D}$ (see Remark \ref{unr}) representing the unramified geometrically integral divisors that are linearly equivalent to $D$ and split geometrically in $Y$.
\end{propdef}

\begin{proof}

As in Proposition \ref{splitnonreg}, let $E'_1, \ldots E'_r$ be the Cartier prime divisors on $Z'$ with support in the exceptional locus of $\pi$, and set $g' := g \circ \pi: Z' \to X$. Let 
\[\psi: \DIV_{Z'/k}^H \to \DIV_{Z'/k}\] 
be the morphism induced by $F' \mapsto \sum_{\bar\sigma \in G/H} \sigma F'$, and for $\underline{d} = (d_1, \ldots, d_r) \in \{0, \ldots, (G\!:\!H)\!-\!1\}^r,$ let
\[\varphi_{\underline{d}}: \DIV_{Z'/k} \to \DIV_{Z'/k}\]  
be the morphism induced by $F' \mapsto F' + \sum_{i=1}^r d_i E'_i$. Using the fact that the pseudoeffective cone has a compact basis (\cite[Proposition 1.3]{BFJ09}, see also Lemma \ref{psef}), we easily see that $\psi$ is proper. Also, as every effective divisor can be written as a sum of effective divisors in finitely many ways, $\psi$ is quasifinite, hence finite. In particular, $\im \psi$ is a closed subset of $\DIV_{Z'/k}$. Therefore, the preimage $W_{\underline{d}}$ of $\im \psi$ under the morphism $\varphi_{\underline{d}} \circ {g'}^* : \mf P'_D \to \DIV_{Z'/k}$ is a closed subset of $\mf P'_D$ for any choice of $\underline{d}$. Set
\[\mf S^Y_D = \bigcup_{\underline{d}} W_{\underline{d}}\] 
with the reduced closed subscheme structure. Then by Proposition \ref{splitnonreg}, $\mf S^Y_D$ represents exactly the unramified geometrically integral divisors that are linearly equivalent to $D$ and split geometrically in $Y$.
\end{proof}

\begin{cor}\label{existD}
 Let $Z, G, X, \Cc , D$ be as above. Then the scheme $\mf D^\mc{C}_{D}$ exists.
\end{cor}

\begin{proof}
Let $H$ be a representative of $\mc{C}$, $Y = Z/H$, $Z \xrightarrow{h} Y \xrightarrow{f} X$ be the corresponding morphisms. Then a prime divisor $D'$ in $X$ has geometric decomposition class $\mc{C}$ if and only if $D'$ splits geometrically in $Y$, but no further extension $Z \to Y' \to Y$, as can be seen by a similar argument to Proposition \ref{splitreg}. Therefore,
\begin{equation}\label{defDbyS}\mf D^\mc{C}_D = \mf S^Y_D \backslash \bigcup_{Z \to Y' \to Y} \mf S^{Y'}_D,\end{equation}
considered as an open subscheme of $\mf S^Y_D$.
\end{proof}

\section{Volume of divisors}\label{secvol}

For the proof of the asymptotic behaviour of $\dim \mf{P}_{mD_0}$ and $\dim \mf{D}^\Cc_{mD_0}$ described in  Theorem \ref{dens1}, we need the theory of the volume of a divisor, which is covered in detail in \cite[Section 2.2.C]{Laz04}. We will need to determine the behaviour of the volume under pull-back and push-forward. This will be done in the following section. Throughout this section, $X$ and $Y$ will be normal projective varieties of dimension $d$ over a given field $k$.

\subsection{Volume}
We first recollect the definition and some basic properties of the volume, referring to \cite[Section 2.2.C]{Laz04} for proofs:

For a Cartier divisor $D$ on $X$, its \emph{volume} is defined to be \[\vol_X(D) = \limsup_{m\to \infty} \frac{h^0(X,mD)}{m^d / d!}.\]
If no confusion can arise, we will write $\vol(D)$ instead of $\vol_X(D)$.
This definition is motivated by the \emph{asymptotic Riemann-Roch theorem} (\cite[Example 1.2.19]{Laz04}):
\begin{prop}\label{asyRR}
For an ample divisor $D$ and any coherent sheaf $\Fc$ on $X$, we have
\[h^0\bigl(X,\Fc \otimes \O_X(mD)\bigr) = \rk(\Fc) \frac{D^d}{d!} m^d + O(m^{d-1}).\]
In particular, $\vol(D) = D^d > 0$.
\end{prop}

The volume increases in effective directions, i.e.\ if $D, E$ are Cartier divisors with $E$ effective, then $\vol(D+E) \geq \vol(D)$. It only depends on the numerical equivalence class of a divisor, and for any Cartier divisor $D$ and $a \in \N$, we have $\vol(aD)=a^d \vol(D)$.
These last properties are used to extend the volume to a function on $\NeSe{X}_\Q := \NeSe{X} \otimes_\Z \Q$. The extended function proves to be continuous with respect to the archimedean topology and is hence extended even further to a continuous function on the N\'{e}ron-Severi space $\NeSe{X}_\R$.

A Cartier divisor $D$ with $\vol(D)>0$ is called \emph{big}. The set of big divisors coincides with the set of Cartier divisors $D$ with $\cln{D}$ in the interior of the pseudoeffective cone $\overline\Eff(X)$.  One of the most important properties of the volume on big divisors is its log-concavity:

\begin{prop}\label{logconcave}
Let $D, D'$ be big Cartier divisors on $X$. Then
\[\vol(D)^{\frac{1}{d}} + \vol(D')^{\frac{1}{d}}\leq \vol(D+D')^{\frac{1}{d}}.\]
If $D + D'$ is ample, then equality holds if and only if $\cln{D}$ and $\cln{D'}$ are proportional, i.e.\ if there are positive integers $n$, $n'$ with $n \cln{D} = n' \cln{D'}$.
\end{prop}

\begin{proof}
In the case $\ch k = 0$, the proof of the inequality is given in \cite[Theorem 11.4.9]{Laz04a}. For ample divisors $D$, $D'$, it is based on the Hodge Index Theorem  \cite[Theorem 1.6.1]{Laz04} which states that for ample divisors $D_1, \ldots, D_d$, one has
\begin{equation}\label{HIT}
(D_1^d) \cdots (D_d^d) \leq (D_1 \cdots D_d)^d.
\end{equation}
The generalization to big divisors is done by using Fujita's approximation theorem (\cite{Fuj94} for $\ch k = 0$, \cite[Theorem 0.1]{Tak07} in the general case).

To prove the statement about the equality case, let us first assume that  $D$ and $D'$ are ample. Then this statement follows from the fact that \eqref{HIT} is an equality if and only if $\cln{D_1}, \cln{D_2}, \ldots, \cln{D_d}$ are proportional, which itself can be derived recursively from the two-dimensional Hodge Index Theorem (\cite[Theorem V.1.9]{Har77}). Therefore, we see that $\vol(\cdot)^{\frac{1}{d}}$ is strictly concave on the ample cone (see also \cite[Corollary E]{BFJ09}).

For general big divisors $D$, $D'$ with $D+D'$ ample and $\cln{D}$ not proportional to $\cln{D'}$, choose an $\varepsilon >0$ such that $\varepsilon D + \frac{1-\varepsilon}{2}(D+D')$ and $\varepsilon D' + \frac{1-\varepsilon}{2}(D+D')$ are ample (this is possible since the ample cone is open in $\NeSe{X}_\R$). Then
\[\begin{split}
\vol(D+D')^{\frac{1}{d}} & > \vol\left(\varepsilon D + \frac{1-\varepsilon}{2}(D+D')\right)^{\frac{1}{d}} + \vol\left(\varepsilon D' + \frac{1-\varepsilon}{2}(D+D')\right)^{\frac{1}{d}}\\
& \geq \varepsilon \vol(D)^{\frac{1}{d}}+ (1- \varepsilon) \vol(D+D')^{\frac{1}{d}} + \varepsilon \vol(D')^{\frac{1}{d}},
\end{split}\]
which is equivalent to $\vol(D+D')^{\frac{1}{d}} > \vol(D)^{\frac{1}{d}} + \vol(D')^{\frac{1}{d}}$.
\end{proof}

\subsection{Behaviour of volume in covers}

We want to investigate the volume of pull-backs and push-forwards of divisors.

\begin{lemma}\label{volpullback}
Let $f:Y \to X$ be a proper, dominant, generically finite morphism of normal projective varieties over $k$. For any $D \in \Div(X)$, we have \[\vol_Y(f^* D) = \deg(f)\vol_X(D).\]
\end{lemma}

\begin{proof}
By the projection formula, we have
\[H^0(Y,\O_Y(m f^* D)) =  H^0(X,f_* (\O_Y(m f^* D)))\cong  H^0(X,(f_*\O_Y)(m D)),\]
so we can restrict our attention to $f_* \O_Y$. There is an open dense subset $U$ of $X$ such that $f_* \O_Y$ is free of rank $n = \deg(f)$, so  $(f_* \O_Y)|_U \simeq \O_U^n$. This isomorphism gives an injection $f_* \O_Y \hookrightarrow \Kc_X^n$, where $\Kc_X$ is the sheaf of total quotient rings of $\O_X$. Set $\Gc = f_* \O_Y \cap \O_X ^n$ and define $\Gc_1$ and $\Gc_2$ by the exact sequences of sheaves
\begin{align*}
0 \to \Gc \to & \, f_* \O_Y  \to \Gc_1 \to 0, \\
0 \to \Gc \to & \,\;\;\O_X^n\;  \to \Gc_2 \to 0.
\end{align*}
The supports of $\Gc_1$ and $\Gc_2$ do not meet $U$, hence have dimension less than $d$. Therefore, $h^0(X, \Gc_i(mD)) = O(m^{d-1})$ for $i=1,2$ by \cite[Proposition 1.31]{Deb01}. Using the long exact cohomology sequence, this implies
\[\begin{split} h^0(Y,\O_Y(m f^* D)) = h^0(X,(f_*\O_Y)(m D)) &= h^0(X,\O_X^n(m D)) + O(m^{d-1})\\ &= n \cdot h^0(X,\O_X(m D)) + O(m^{d-1}),\end{split}\]
from which the assertion follows.
\end{proof}

\begin{prop}\label{volpushforward}
Let $f: Y \to X$ be a finite branched cover of normal projective varieties over $k$ of dimension $d$. For any $E \in \Div(Y)$, we have
\begin{equation}\label{eqvolpf}
\vol_Y(E) \leq \frac{1}{\deg(f)^{d-1}} \vol_X(D),
\end{equation}
where $D=f_* E$. If $D$ is ample, then equality holds if and only if $f^* D \equiv \deg(f) E$.
\end{prop}

\begin{proof}
Take the Galois closure $g:Z \to X$ of $f: Y \to X$ (see Remark \ref{makeGal}); let $G$ and $H$ denote the Galois groups of  $g:Z \to X$ and $h: Z \to Y$, respectively.

We have $\vol_Z(h^* E) = \vol_Z(\sigma h^* E) \foral \sigma \in G$, so using \ref{pfthrupb}, \ref{logconcave} and \ref{volpullback} we get
\[\begin{split}
(\# G \cdot \vol_X(D))^{1/d} &= \vol_Z(g^* D)^{1/d}
= \vol_Z\Biggl(\, \sum_{\bar  \sigma \in G/H} \sigma h^*E \Biggr)^{\!\!1/d}  \\
&\geq \sum_{\bar  \sigma \in G/H} \vol_Z\Bigl(  \sigma h^*E \Bigr)^{1/d}
=  (G:H) \vol_Z\bigl(h^*E\bigr)^{1/d}\\
& = (G:H) \cdot \bigl(\# H \cdot \vol_Y(E)\bigr)^{1/d}.
\end{split}\]
Taking $d$th powers, the inequality follows.

We are left to consider the equality case. Proposition \ref{logconcave} implies that we have equality in \eqref{eqvolpf} if and only if all $\cln{\sigma h^* E}$ are proportional. As all of them have the same image under $g_*$, they have to be equal in this case. This is only true if $g^* D \equiv (G:H) h^* E$ and thus $\# H \cdot f^* D = h_* g^* D \equiv \deg(f) \# H \cdot E$. The converse is obvious.
\end{proof}

\subsection{Applications of the volume}

Even though the last propositions will be useful in the proof of Theorem \ref{dens1}, we still need more refined versions and applications of the theory above. We first investigate an equivariant version of the volume.

\begin{lemma}
Let $Z$ be an integral projective variety over $k$ of dimension $d$, and let $H$ be a finite subgroup of $\Aut_k(Z)$. Let $F$ be an effective big $H$-invariant Cartier divisor on $Z$. Then
\[\lim_{m \to \infty} \frac{\dim_k H^0(Z,\O_{Z}(mF))^H}{m^d /d!} = \frac{\vol(F)}{\# H}.\]
\end{lemma}

\begin{proof}
 In characteristic 0, this statement has been proven by Paoletti~(\cite[Theorem 1]{Pao05}); it can also be derived from a result of Howe (\cite{How89}). For arbitrary characteristic, let us first consider the case when
 \[S:=\bigoplus_{m\geq 0} S_m :=\bigoplus_{m\geq 0} H^0(Z,\O_{Z}(mF))\]
 is finitely generated as a $k$-algebra\footnote{This is the case, e.g., if $F_0$ is semiample, see \cite[Example 2.1.30]{Laz04}.}, and set
 \[R := S^H= \bigoplus_{m\geq 0} H^0(Z,\O_{Z}(mF))^H,\]
 $R_m$ being the $m$th degree term. Then by a theorem of Symonds~(\cite[Theorem 1.2]{Sym00}), there exists a free $kH$-submodule $M$ of rank 1, a sum of homogeneous pieces, such that the product map $R \otimes_k M \to S$ is injective and such that
 \[\lim_{m\to \infty}\frac{\dim_k (RM)_m}{\dim_k S_m} = 1,\]
 where $(RM)_m$ denotes the $m$th order homogeneous summand of the subalgebra $RM \subset S$. As $F$ is $H$-invariant and effective, $\dim_k R_m$ is strictly increasing; therefore, this statement implies that \[\lim_{m \to \infty} \frac{\dim_k R_m}{\dim_k S_m}=\frac{1}{\dim_k M}=\frac{1}{\# H},\] and the assertion follows using the definition of the volume.
 In the general case, we exhaust $S$ by finitely generated subalgebras $S^{(1)} \subset S^{(2)} \subset \ldots \subset S$ such that \[\vol(S^{(i)}) \xrightarrow{i \to \infty} \vol(S)=\vol(F),\] where for a graded $k$-algebra $A = \bigoplus_{m \geq 0} A_m$ we set $\vol(A)= \lim_{m \to \infty} \frac{\dim_k A_m}{m^d/d!}$. Our results for the $S^{(i)}$ imply the one for $S$.
\end{proof}

So far, it seems as if the volume only describes the asymptotic behaviour of $h^0(X,mD)$ for large $m$. But in fact, it can be used to describe $h^0(X,D')$ for `large' divisors $D'$ no matter whether they are of the form $mD$ with $m \gg 0$ or not. Before stating this more formally, we have to fix our context.

\begin{cont}\label{ContS}
Let $Z$ be a normal projective variety over $k$ of dimension $d$, $H$ a finite subgroup of $\Aut_k(Z)$. Then $H$ also acts on the N\'{e}ron-Severi space $\NeSe{Z}_\R$; let $\NeSe{Z}_\R^H$ denote the $H$-invariant subspace.

Let $W$ be an affine subspace of $\NeSe{Z}_\R^H$ not containing the origin, $S$ a nonempty compact convex subset of $W$ with $v_{max} := \max \{\vol(s) | s \in S\} > 0$. Furthermore, let $C_S$ be the cone in $\NeSe{Z}_\R^H$ generated by $S$, and define the function $m : C_S \to \R_{\geq 0}$ by $c_s \mapsto$ the unique real number $r$ such that $c_s = r s$ with $s \in S$.
\end{cont}

\begin{prop}\label{Hvol}
In the above context, the following holds: Given any $\varepsilon >0$, there exists a positive integer $M_\varepsilon$ such that for every line bundle $\Lc$ with $\cln{\Lc}\in C_S$ and $m_\Lc := m(\cln{\Lc}) \geq M_\varepsilon$, we have
\[\dim H^0(Z, \Lc)^H < \frac{v_{max}(1+\varepsilon)}{\# H} \frac{m_\Lc^d}{d!}.\]
\end{prop}

\begin{proof}
The proof is indirect: Assume there exist line bundles $\Lc_i$ with $\cln{\Lc_i} \in C_S$, $m_i = m(\cln{\Lc_i})  \xrightarrow{i \to \infty} \infty$ and
\[\dim H^0(Z, \Lc_i)^H \geq \frac{v_{max}(1+\varepsilon)}{\# H} \frac{m_i^d}{d!}.\]
All of these $\Lc_i$ are necessarily $H$-invariant. By carefully changing the $\Lc_i$ (in effective directions), we will derive a contradiction to the preceding lemma. This will be done in several steps.

\emph{Step 1:} We first change $W$ and $S$. By reducing to an affine subspace if necessary, we can assume that $S$ contains an open subset of $W$; wiggling $W$ and $S = W \cap C_S$ (and decreasing $\varepsilon$) a little bit, we may assume that $W \cap \NeSe{Z}_\Q$ is dense in $W$. Since all $\cln{\Lc_i}$ lie in the pseudoeffective cone, we may replace $S$ and $C_S$ by their intersections with the pseudoeffective cone; therefore, the interior $S^\circ$ of $S$ lies in the big cone.

\emph{Step 2:} Choose a big rational numerical divisor class $\eta' \in S^\circ \cap \NeSe{Z}_\Q$. Then for some integer $m' \gg 0$, all line bundles mapping to $m' \eta'$ are effective (see \cite[Lemma 2.2.42]{Laz04}). Replacing $m'$ by $m' \# H$, we can find a line bundle $\Lc'$ with $\cln{\Lc'} = m' \eta'$ that carries an effective $H$-invariant Cartier divisor; in fact every line bundle in $\Lc' N_H (\Pic^\tau(Z))$ does.

The sequence $\frac{1}{m_i}\cln{\Lc_i}$ has an accumulation point in $S$. After restricting to a subsequence, replacing $\varepsilon$ by a smaller positive number $\varepsilon'$ and the $\Lc_i$ by $\Lc'_i = \Lc_i \otimes \Lc'^{\otimes \lfloor\frac{m_i}{m'}\alpha\rfloor}$ with $1 < 1+ \alpha < \sqrt[d]{\frac{1+\varepsilon}{1+\varepsilon'}}$, we can and will assume that the  $\frac{1}{m_i}\cln{\Lc_i}$ converge to a point $\eta$ in $S^\circ$.

\emph{Step 3:} We choose linear independent big rational divisor classes $\beta_0, \ldots, \beta_r$ in $S^\circ \cap \NeSe{Z}_\Q$ with $r=\dim W$ such that $\eta$ lies in the interior of the simplex $S_\beta$ with vertices $\beta_j$. After replacing $S$, $W$ and the $\beta_i$ by large enough multiples, we can assume that for all $j= 0, \ldots r$, there exists an effective $H$-invariant Cartier divisor $B_j$ with $\cln{B_j}=\beta_j$ (using the same argument as for the $\Lc'$ above).

After throwing away finitely many $\Lc_i$, we can assume that all $\frac{1}{m_i}\cln{\Lc_i}$ lie in $S_\beta$, i.e.\ for every $i > 0$, we have \[\frac{1}{m_i}\cln{\Lc_i} = \sum_{j=0}^r a_{ij} \beta_j \quad \text{for some } a_{ij} \in \Q_{\geq 0} \text { with } \sum_{j=0}^r a_{ij}=1.\]
Since $\eta$ lies in the interior of $S_\beta$, we can assume that $a_{ij} \geq \delta$ for all $i > 0$, $j = 0, \ldots, r$ and some fixed $\delta > 0$. Furthermore, since the $\frac{1}{m_i}\cln{\Lc_i}$ form a Cauchy sequence, we may assume that $|a_{ij} - a_{i'\!j}| < \delta^2$ for all $i, i' > 0$, $j=0, \ldots r$.

\emph{Step 4:} Let $V \subset \NeSe{Z}_\R^H$ be the linear subspace spanned by $W$. Then $\NeSe{Z} \cap V$ is a lattice $\Lambda$ in $V$ that contains $\Lambda' := \bigoplus_j \Z \beta_j$ as a sublattice of finite index. Therefore, after possibly changing to a subsequence, we can assume that all $\cln{\Lc_i}$ lie in the same class of $\Lambda/\Lambda'$. In other words, for any two line bundles $\Lc_i$, $\Lc_{i'}$ in this sequence, we have $m_i a_{ij}- m_{i'} a_{i'\!j} \in \Z$ for all $j=0, \ldots, r$.

Also, using the fact that $\Pic(Z)^H/N_H \Pic(Z)$ is finite, we can restrict to a subsequence one last time and assume that all the $\Lc_i$ lie in the same class modulo $N_H \Pic(Z)$.

\emph{Step 5}: Finally, we are able to put all pieces together. Choose $\tilde \Lc$ to be one of the $\Lc_i$, define $\tilde m, \tilde a_0, \ldots, \tilde a_r \in \Q$ by
\[\frac{1}{\tilde m}\cln{\tilde\Lc} = \sum_{j=0}^r \tilde a_j \beta_j \in S_\beta\]
and fix $N:= \# H (\Lambda:\!~\!\Lambda')$. For any integer $p_i$ with
\[(1+p_i N) \tilde m \tilde a_j > m_i a_{ij} \quad \foral j=0, \ldots, r,\]
the sheaf $\tilde \Mc_i = \tilde\Lc^{\otimes (1+p_i N)}\!\otimes\!\Lc_i^\vee$ is numerically equivalent to $\O_Z(\sum_{j=1}^r b_{ij} B_j)$ for some nonnegative integers $b_{ij}$; furthermore, these two line bundles differ by an element of $N_H \Pic^\tau(Z)$. Looking back at our remarks concerning the construction of $\Lc'$ in Step 2, we can even assume that $\tilde \Mc_i = \O_Z(\sum_{j=0}^r b_{ij} B_j)$, hence $H^0(Z, \tilde \Mc_i)^H \neq 0$. In particular,
\[\dim H^0\bigl(Z, \tilde\Lc^{\otimes (1+p_i N)}\bigr)^H \geq \dim H^0(Z, \Lc_i)^H \geq \frac{v_{max}(1+\varepsilon)}{\# H} \frac{m_i^d}{d!}.\]
For every $i$, we choose the minimal such $p_i$, i.e.\
\[p_i = 1+ \max \biggl\{ \biggl\lfloor \frac{m_i a_{ij}-\tilde m \tilde a_j}{N \tilde m \tilde a_j}\biggr\rfloor\ \biggl|\ j=1, \ldots, r\biggr\}.\]
Since $a_{ij}\leq \tilde a_j+\delta^2 \leq \tilde a_j(1+\delta)$, we have

\[1+p_i N \leq N+ \max \Bigl\{ \frac{m_i a_{ij}}{\tilde m \tilde a_j} \Bigr| j=0, \ldots, r \Bigr\} \leq N + \frac{m_i}{\tilde m}(1+\delta) \leq \frac{m_i}{\tilde m}(1+2\delta) \quad\text{ for } i \gg 0,\]
hence
\[\dim H^0(Z,  \tilde\Lc^{\otimes (1\!+\!p_i N)}\!)^H \geq \frac{v_{max}(1\!+\!\varepsilon)}{\# H \cdot d!} {\left(\frac{(1\!+\!p_i N) \tilde m}{1\!+\!2\delta}\right)\!\!}^{d} = \frac{v_{max}(1\!+\!\varepsilon)\tilde m^d}{\# H (1\!+\!2\delta)^d} \frac{(1\!+\!p_i N)^d}{d!}.\]
Choosing $\delta > 0$ small enough and using the lemma above, this implies
\[\vol(\tilde \Lc) \geq \frac{v_{max}(1+\varepsilon)\tilde m^d}{(1+2\delta)^d} > v_{max} \tilde m^d.\]
But since $\frac{1}{\tilde m}\cln{\tilde \Lc} \in S$, $\vol(\tilde \Lc) = \tilde m^d \vol(\frac{1}{\tilde m}\cln{\tilde \Lc}) \leq v_{max} \tilde m^d$, contradiction.
\end{proof}

\begin{rem}\label{Hvolrem}
To give one of the main examples of Context \ref{ContS}, fix an ample class $\eta \in \NeSe{Z}^H_\R$ and define $\deg = \deg_\eta : \NeSe{X}_\R \to \R$ by $\deg(\delta) = \eta^{d-1}\cdot \delta$.

Then we can set $S = \{\delta \in \overline\Eff(Z)\cap \NeSe{X}^H_\R \ | \ \deg(\delta)=1\}$. In this case, we have $C_S = \overline\Eff(Z) \cap \NeSe{X}^H_\R, m = \deg_{|C_S}$ and $v_{max} = \vol(\eta)^{1-d}$. This follows from the following lemma:
\end{rem}

\begin{lemma}\label{degree}
Let $Z$ be a normal projective variety of dimension $d \leq 2$, let $\eta$ be an ample class in $\NeSe{Z}_\R$ and $\deg=\deg_\eta$. Then the following holds:
\begin{enumerate}[(a)]
\item For every $\delta \in \overline\Eff(Z)$, $\delta \neq 0$, we have $\deg(\delta) > 0$.
\item The set $S = \{\delta \in \overline\Eff(Z) | \deg(\delta) = 1\}$ is compact.
\item For a big class $\delta$, we have $\vol(\delta) \leq \frac{\deg(\delta)^{d}}{\vol(\eta)^{d-1}}$, with equality if and only if $\delta$ is proportional to $\eta$.
\end{enumerate}
\end{lemma}

\begin{proof}
\begin{enumerate}[(a)]
\item For effective classes, this is the easy direction of the Nakai-Moishezon criterion (\cite[Theorem 1.2.23]{Laz04}); by linearity and continuity, this implies that $\deg$ is non-negative on $\overline\Eff(Z)$. Assume that for some $\delta \in \overline\Eff(Z)$, we have $\deg(\delta)= \eta^{d-1} \cdot \delta = 0$. Then for some $\varepsilon >0$,  $\eta + t \delta$ is ample for all $|t| < \varepsilon$. Thus the function $(\eta + t\delta)^{d-1} \cdot \delta$ has a local minimum at $t=0$, hence $\eta^{d-2} \cdot \delta^2 =0$.
This implies $\delta=0$ by \cite[9.6.3 (h)$\Rightarrow$(b)]{Kle05}.
\item Assume the converse: Let $(\delta_i)_{i \in \N}$ be a sequence in $S$ with $|\delta_i| \xrightarrow{i\to \infty} \infty$ for some norm $|\cdot|$ on $\NeSe{Z}_\R$. The sequence $\left(\frac{1}{|\delta_i|} \, \delta_i\right)_{i \in \N}$ lies on the (compact) unit sphere, so has an accumulation point $\tilde \delta$, which lies in $\overline\Eff(Z)$ since the latter is closed and convex. But on the other hand, we have $\tilde\delta \neq 0$ and $\deg(\tilde\delta)=0$ by construction, in contradiction to (a).
\item For ample classes $\delta$, this follows from \eqref{HIT} in the proof of Proposition \ref{logconcave}; the extension to big classes is done in the same way as in the proof of Proposition \ref{logconcave}.\qedhere
\end{enumerate}
\end{proof}

\begin{cor}\label{psef}
The pseudoeffective cone of a normal projective variety $Z$ has a compact basis, i.e.\ for $\delta \in \overline\Eff(Z)$, the set $\{\delta' \in \overline\Eff(Z)  \ | \ \delta - \delta' \in \overline\Eff(Z)\}$ is compact.
\end{cor}

\begin{proof}
For any ample $\eta \in \NeSe{X}_\R$, the closed set $\{\delta' \in \overline\Eff(Z)  \ | \ \delta - \delta' \in \overline\Eff(Z)\}$ is contained in the compact set $\{\delta' \in \overline\Eff(Z)  \ | \ \deg_\eta(\delta') \leq \deg_\eta(\delta)\}$ and hence compact itself.
\end{proof}

\begin{rem}
For smooth varieties, an alternative proof is given in \cite[Proposition 1.3]{BFJ09}.
\end{rem}

\begin{cor}\label{psefcor}
Let $Z$ be a normal projective variety, $\eta \in \NeSe{Z}_\R$ be an ample class. Then for any $C > 0$, there are only finitely many classes $\delta \in \overline\Eff(Z) \cap \NeSe{Z}$ with $\deg_\eta(\delta) \leq C$.
\end{cor}

\section{Asymptotics of the density}\label{secasy}

In the situation of Theorem \ref{dens1}, we have proved the existence of the schemes $\mf P_D$ and $\mf D^\Cc_D$. We now want to consider the asymptotics of $\dim \mf P_{mD_0}$ and $\dim \mf D^\Cc_{mD_0}$.

\subsection{Asymptotic behaviour of $\dim \mf P_{mD_0}$}

For any Cartier divisor $D$ on $X$, $\mf{P}_D$ is an open subscheme of $\mf{L}_D$, so its dimension is $h^0(X,D)-1$ unless it is empty.

\begin{prop}[Bertini's theorem]\label{bertini}
Assume $\dim X > 1$. Let $D$ be a very ample Cartier divisor on $X$. Then $\mf{P}_D$ is nonempty.
\end{prop}

\begin{proof}
It is enough to prove that $\mf{P}_D(\bar k)$ is nonempty. But this corresponds to the set of integral divisors in the linear system $|D_{\bar k}|$. By Bertini's theorem, the generic member of $|D_{\bar k}|$ is irreducible (\cite[Theorem 3.4.10]{FOV99}) and reduced (\cite[Corollary 3.4.14]{FOV99}), hence integral.
\end{proof}

\begin{cor}\label{bertinicor}
Assume $\dim X \geq 2$, and let $D$ be an ample divisor and $D'$ an arbitrary divisor on $X$. Then \[\dim \mf P_{mD+D'} = \frac{\vol(D)}{d!}\cdot m^d + O(m^{d-1}).\]
\end{cor}

\begin{proof}
For $m \gg 0$, $mD+D'$ is very ample, so $\dim \mf P_{mD+D'}=h^0(X,mD+D')-1$ by Proposition \ref{bertini} and the remarks before. Proposition \ref{asyRR} then implies our claim.
\end{proof}

\subsection{Asymptotic behaviour of $\dim \mf D^\Cc_{mD_0}$}\label{secupbd}
From our proof for the existence of $\mf D^\mc{C}_D$ in Corollary \ref{existD}, it is clear that describing the asymptotic behaviour of $\dim \mf D^\mc{C}_{mD_0}$ will come down to describing the asymptotic behaviour of $\dim \mf S^{Z/H}_{mD_0}$, with $H$ any representative of $\mc{C}$.

As before, we fix such an $H$, set $Y = Z/H$ and denote the canonical morphism $Y \to X$ by $f$. The following two lemmata will give upper and lower asymptotic bounds on $\dim \mf S^Y_{mD_0}$, which finally enable us to finish the proof of Theorem \ref{dens1}.

\begin{lemma}\label{upperbound}
In the above context, fix a divisor $D_0$ on $X$. Then
\[\limsup_{m\to\infty} \frac{\dim \mf S^Y_{mD_0}}{m^d/d!} \leq \frac{\vol(D_0)}{(\deg f)^{d-1}}.\]
\end{lemma}

\begin{proof}
We can assume $k$ to be algebraically closed. Looking back at the proof of Proposition \ref{existS} and using its notation, we see that for any Cartier divisor $D$ on $X$, $\dim \mf S^Y_{D}$ is bounded by the maximum of the dimensions of the schemes $\im \psi \cap \mf L_{{g'}^*D+\sum_i d_i E'_i}$. Using that $\psi$ is finite, we get
\[\dim \mf S^Y_{D} \leq \max \biggl\{ \dim \Bigl(\DIV_{Z'/k}^\beta\Bigr)^{\!H} \biggr|\,\begin{array}{c}\beta \in \NeSe{Z'}^H\!,\, \psi(\beta) = \cln{{g'}^{\!*}D+\textstyle\sum_i d_i E'_i} \\ \text{with } 0 \leq d_i < (G:H)\end{array} \biggr\}. \]
 Set $\eta' = \cln{{g'}^* D_0}$. Then for any $\beta$ occurring in the inequality above, we have
\[(G:H) \cdot \deg_{\eta'}(\beta) = \deg_{\eta'}(\psi(\beta)) = \deg_{\eta'}({g'}^*D)+\sum_{i=1}^r d_i \deg_{\eta'}(E'_i).\]
With $\eta=\cln{D_0}$, one has $\deg_{\eta'}({g'}^* D)=\# G \cdot \deg_\eta(D)$ by \cite[Proposition 1.10]{Deb01}. Set $C = \sum_i \deg_{\eta'}(E'_i)$. Then
\[\# H \cdot \deg_\eta(D) \leq \deg_{\eta'}(\beta) \leq \# H \cdot \deg_\eta(D) + C.\]
Hence,
\begin{equation}\label{equp1}
\;\dim \mf S^Y_{D} \leq \max \biggl\{ \dim \Bigl(\DIV_{Z/k}^\beta\Bigr)^{\!H} \biggr| \, \beta \!\in\! \NeSe{Z}^H,\, \deg_{\eta'}(\beta)\! \leq\! \# H \!\cdot \deg_\eta(D) + C \biggr\}.
\end{equation}

For a fixed class $\beta$, consider the Abel map $W := (\DIV_{Z/k}^{\beta})^H \to (\PIC_{Z/k}^{\beta})^H=: V$. Then by \cite[Corollaire 5.6.7]{Gro65}, we have
\[ \dim W \leq \dim V + \max \{\dim W_v \ | \ v \in V \}. \]
Because of the upper semicontinuity of the dimension of the fiber (\cite[Corollaire 13.1.5]{Gro66}), it is enough to take the maximum only over the closed points of $V$. But the closed points correspond to invertible sheaves $\Lc$ with numerical equivalence class $\beta$, and the corresponding fibers $W_\Lc$ are just the schemes $\mf L_{\Lc}^H$, hence have dimension $\dim H^0(Z,\Lc)^H-1$. Furthermore, $\dim V \leq \dim \PIC_{Z/k}^{\beta} = \dim \PIC_{Z/k}^{\tau}$. Therefore
\begin{equation}\label{equp2}
\dim \left(\DIV_{Z/k}^{\beta}\right)^H \leq \dim \PIC_{Z/k}^\tau -1 + \max \Bigl\{ \dim H^0(Z,\Lc)^H \ \Bigl| \ \cln{\Lc}=\beta \Bigr\}
\end{equation}
which applied to \eqref{equp1} gives
\[
\dim \mf S^Y_{D} \leq \dim \PIC_{Z/k}^\tau + \max \biggl\{ \dim H^0(Z,\Lc)^H \ \biggl| \begin{array}{c}\Lc \in \Pic(Z)^H,\quad\\ \deg_{\eta'}(\Lc) \leq  \# H \cdot \deg_\eta(D) + C\end{array} \biggr\} .
\]
Applying Proposition \ref{Hvol} in the way described in Remark \ref{Hvolrem}, we derive
\[\limsup_{m \to \infty} \frac{\dim \mf S^Y_{m D_0}}{m^d / d!} \leq \lim_{m \to \infty} \frac{v_{max} \bigl(\# H \!\cdot \deg_\eta(m D_0) + C\bigr)^d}{\#H \cdot m^d}= v_{max} (\# H)^{d-1} \vol(D_0)^d\]
with $v_{max}=\vol({g'}^* D_0)^{1-d}=(\# G \cdot \vol(D_0))^{1-d}$. This implies the assertion.
\end{proof}

This gives an upper bound. For the lower bound, we use the following

\begin{lemma}\label{lowerbound}
Let $D_0$ be an ample Cartier divisor on $X$, $E$ be a Cartier divisor on $Y$ with $f_* E \sim D_0$. Set $E_m = \lfloor \frac{m}{n} \rfloor f^* D_0 + (m -\lfloor\frac{m}{n}\rfloor n)E$, where $n = \deg(f)$. Then for all $m \gg 0$, the image of the map \[\mf f_*: \mf L_{E_m} \to \mf L_{mD_0}\] intersects nontrivially with $\mf P_{mD_0}$; in particular, we have \[\dim \mf S_{mD_0}^Y \geq \dim \mf L_{E_m} = \frac{\vol(D_0)}{n^{d-1}} \frac{m^d}{d!} + O(m^{d-1}).\]
\end{lemma}

\begin{proof}
We can assume that $k$ is algebraically closed. For the first assertion, it is enough to show that $(\mf f_* \mf L_{E_m} \cap \mf P_{mD_0})(k) \neq \emptyset$. To do this, set $\mf V_m = \mf P_{E_m} - \bigcup_{Y'} \mf S^{Y'}_{E_m}$, where $Y'$ runs through all of the finitely many intermediate covers of $Z \to Y$ apart from $Y$ (including $Z$). By our definitions, $\mf V_m(k)$ consists of all Cartier prime divisors $E'$ linearly equivalent to $E_m$ that do not split in any of those intermediate covers. By Hilbert's decomposition theory, this is equivalent to saying that $h^* E'$ is a prime divisor. $\mf V_m$ is an open subset of $\mf L_{E_m}$ and, for $m \gg 0$, our earlier dimension calculations show that $\mf V_m$ is, in fact, dense in $\mf L_{E_m}$.

Assume $\mf f_* \mf V_m \cap \mf P_{mD_0} = \emptyset$. For $E' \in \mf V_m(k)$, we have $f_* \WD{E'} = r \cdot f(\WD{E'})$ with $r = [K(\WD{E'}):K(f(\WD{E'}))] = (G_{h^* E'} : H_{h^* E'}) = (G_{h^* E'}: H)$, so $f_* E' \notin \mf P_{mD_0}(k)$ if and only if $G_{h^* E'} \supsetneq H$. Therefore, $h^*$ maps $\mf V_m$ into the union of the subsets of $\mf L_{h^*E_M}$ consisting of all $H'$-invariant divisors, where $H'$ runs through all groups $H \lneq H' \leq G$. In fact, since all these subsets are closed and $\mf V_m$ is dense in the irreducible scheme $\mf L_{E_m}$, there exists a subgroup $H'$ of $G$ with $H \lneq H'$ such that $h^* E'$ is $H'$-invariant for all $E' \in \mf L_{E_m}$. Thus the difference of any two divisors in $|E_m|$ is $\divd(u)$ with $u \in K(Z)^{H'}$; on the other hand, since $E_m$ is very ample for $m \gg 0$, the set of these $u$ generates $K(Y)$. So we get $K(Y) \subseteq K(Z)^{H'}$ and by Galois correspondence $H \geq H'$, which contradicts $H \lneq H'$.

 Since we know that $\mf U_m = \mf ({f_*}|_{\mf{L}_{E_m}})^{-1} \mf P_{mD_0} \subset \mf L_{E_m}$ is nonempty and open, it is a dense open subset of $\mf L_{E_m}$. Because $\mf f_* : \mf L_{E_m} \to \mf L_{mD_0}$ is finite, we get
 \[\dim \mf L_{E_m} = \dim \mf U_m = \dim (\mf f_* \mf L_{E_m} \cap \mf P_{mD_0}) \leq \dim \mf S^Y_{mD_0}.\]
 The rest follows from Proposition \ref{asyRR} and from the fact that $\vol(f^* D_0) = n \vol (D_0)$ (Lemma \ref{volpullback}).
 \end{proof}

\begin{rem}\label{remratpts}
The proof of the lemma also implies that for the case of an infinite field $k$, the sets $\mf S^Y_{mD_0}(k)$ and $\mf D^\mc{C}_{mD_0}(k)$ are nonempty (and indeed infinite) for $m \gg 0$. To see this, one just has to point out that $\mf V_m \cap \mf U_m$ is a nonempty open subset of the projective space $\mf L_{E_m}$, hence contains infinitely many $k$-rational points. The images of these points under $\mf f_*$ are $k$-rational points of $\mf S^Y_{mD_0}$, and all of those will have $\Cc$ as both decomposition group and geometric decomposition group by construction.

(In the case where $k$ is a finite field, $\mf S^Y_{mD_0}$ and $\mf D^\mc{C}_{mD_0}$ still have $k$-rational points for $m \gg 0$, but one has to use counting arguments instead of dimension arguments. This will be done in subsection \ref{secfin}.)
\end{rem}

From the bounds for the asymptotics of $\dim \mf S^Y_{mD_0}$, we can derive our main theorem instantly:

\begin{proof}[Proof of the Theorem \ref{dens1}]
From Lemma \ref{upperbound} and Lemma \ref{lowerbound}, we immediately get that
\[\frac{\dim \mf S^Y_{mD_0}}{m^d/d!} \xrightarrow{m \to \infty} \frac{1}{\deg(f)^{d-1}} \vol(D_0)\]
for any finite branched cover $f: Y \to X$ of $X$ with $Y$ normal and geometrically integral. Applying this fact to \eqref{defDbyS} in Corollary \ref{existD}, we get that for $m \gg 0$, $\dim \mf D^\Cc_{mD_0} = \dim \mf S^{Z/H}_{mD_0}$ for any representative $H$ of $\mc{C}$, so
\[\lim_{m \to\infty}\frac{\dim \mf D^{\Cc}_{mD_0}}{m^d/d!} = \frac{1}{(G:\Cc)^{d-1}} \vol(D_0).\]
Combining this with Corollary \ref{bertinicor}, we are done with the main part of the theorem.

If $D_0$ is not the push-forward of a Cartier divisor on $Z$, then $\# G \cdot D_0 = g_* g^* D_0$ still is. Also, the upper bounds hold without this restriction on $D_0$. This implies the remaining part of the theorem.
\end{proof}

\section{Modifications and Applications}\label{secapp}

As mentioned in the introduction, there are several ways to modify Theorem \ref{dens1}. We will present three different extensions, followed by an alternative proof of the Bauer-Schmidt theorem.

\subsection{Polynomial behaviour of $\dim \mf D^\Cc_{mD_0}$.}

In the main theorem, we have only used the asymptotic behaviour of $\dim \mf P_{mD_0}$ and $\dim \mf D^\Cc_{mD_0}$. But at least for $\dim \mf P_{mD_0}$, we know more than just that: For $m \gg 0$, the function $m \mapsto \dim \mf P_{mD_0}=h^0(X,mD_0)-1$ is polynomial, and its coefficients can be determined using the Hirzebruch-Riemann-Roch formula (\cite[Appendix A, Theorem 4.1]{Har77}). Thus, one can expect $\dim \mf D^\Cc_{mD_0}$ to behave polynomially as well. This is indeed true, at least under some further assumptions on $Z$ and $X$:

\begin{theo}\label{dens2}
In the situation of Theorem \ref{dens1}, assume furthermore that $Z$ and $X$ are regular and $g: Z \to X$ is \'{e}tale.
We assume that $D_0$ is an ample Cartier divisor that is linearly equivalent to the push-forward of a Cartier divisor on $Z$. Then there are polynomials $Q_0(t), \ldots, Q_{(G:\Cc)-1}(t) \in \Q[t]$ such that
\[\dim \mf D^\Cc_{mD_0} = Q_{r(m)}(m) \quad \foral m \gg 0,\]
where $r(m) = m \!\!\mod (G:\Cc) \in \{0, \ldots, (G:\Cc)-1\}$.
\end{theo}

\begin{proof}
As in the proof of the main theorem, it is enough to prove that $\dim \mf S^Y_{mD_0}$ has the claimed behaviour, where $Y = Z/H$, $H$ a representative of $\Cc$, $f: Y \to X$ the corresponding cover.

By Lemma \ref{splitWeil}, we know that in this case, $\mf S^Y_{mD_0} = \mf f_* \DIV_{Y/k} \cap \mf P'_{mD_0}$. In this formula, we can even replace $\DIV_{Y/k}$ by the union of the $\DIV_{Y/k}^\beta$ for those $\beta$ that fulfill $\beta - \frac{m}{\deg(f)} \cln{f^* D_0} \in \ker (f_*: \NeSe{Y}_\R \to \NeSe{X}_\R$).

\noindent Let $\beta_1, \ldots, \beta_s \!\in\! \NeSe{Y}$ be a basis of $\ker f_*$, and set \mbox{$\eta \!=\! \frac{1}{\deg(f)}\cln{f^* D_0} \!\in\! \NeSe{Y}_\R$}. Let $W$ be the affine space containing $\eta$ and spanned by $\beta_1, \ldots, \beta_s$, and set
\[S = \left\{\eta + \sum_{i=1}^s t_i \beta_i \ | \ \underline t=(t_i)_i \in \R^s, |\underline t|\geq c\right\}\cap \overline\Eff(Y)\]
for some fixed $c>0$. By Proposition \ref{Hvol}, Remark \ref{Hvolrem} and Proposition \ref{volpushforward}, we have
\[\limsup_{m \to \infty} \frac{\max \bigl\{ \dim \DIV_{Y/k}^\beta\ \bigl| \ \beta \in \NeSe{Y} \cap m S\bigr\}}{m^d/d!} \leq \max \{ \vol(s) \ | \ s\in S\} < \vol(\eta).\]
Since $\vol(\eta) = \deg(f)^{1-d} \vol(D_0) = \limsup_{m \to \infty} \frac{\mf S^Y_{mD_0}}{m^d/d!}$, we can restrict our considerations to the $\beta \in \NeSe{Y}$ of the form $m \eta + \sum_i t_i \beta_i$ with $|\underline t| < cm$. For $c$ small enough, we can assume that any such class $\beta$ is ample.

For any of those $\beta$, let us take a closer look at the line bundles representing it: Using Fujita's vanishing theorem \cite[Theorem 1.4.35]{Laz04}  and the fact that the Euler characteristic is invariant under numerical equivalence, we get that $h^0(Y, \Lc)=\chi(Y,\Lc)$ is the same for all $\Lc$ with $\cln{\Lc}=\beta$ (we might have to decrease $c$). Hence we have \[\dim \left(\mf f_* \DIV^\beta_{Y/k} \cap \mf P'_{mD_0}\right) = \chi(X,\beta)-1 + \dim \ker \bigl(\mf f_*: \PIC^\tau_{Y/k} \to  \PIC^\tau_{X/k}\bigr),\]
where $\chi(Y, \beta)=\chi(Y,\Lc)$ for any line bundle $\Lc$ representing $\beta$.

Therefore, it is enough to show that
\[\max \Bigl\{ \chi(Y,\beta) \ \Bigl| \ \beta \in \NeSe{Y},  \, \beta=m\eta+\textstyle\sum_i t_i \beta_i, \, |\underline t|\leq cm \Bigr\}\]
has the asserted polynomial behaviour. By assumption, there exists a Cartier divisor $E_0$ on $Y$ whose push-forward is linearly equivalent to $D_0$. Hence, every divisor class on $Y$ mapping to $m\cln{D_0}$ can be written as $m \cln{E_0}+\sum_i t'_i \beta_i$ with some $t'_i \in \Z$, and by Hirzebruch-Riemann-Roch or simply \cite[Theorem 1.5]{Deb01},
\[\chi\Bigl(Y,m \cln{E_0}+\textstyle\sum_i t'_i \beta_i\Bigr)\]
behaves polynomially in $m$ and $\underline t'\!$. Since $\cln{E_0}\!\!-\eta \!=\! \sum_i \delta_i \beta_i$ for some $\underline\delta \in \frac{1}{\deg(f)} \Z^s$, we can make a transformation of variables and get that $\chi(Y, m\eta+\sum_i t_i \beta_i)$ is a polynomial in $m$ and $\underline t$ where defined, i.e.\ for $m \in \Z$ and $\underline t - m \underline \delta \in \Z^s$.

We want to use Lemma \ref{polymax} below on $P(m, \underline t)=\chi(Y, m\eta+\sum_i t_i \beta_i)$. In order to show that the conditions of the lemma are fulfilled, we have to consider the highest degree homogeneous part, which by definition is $(m\eta+\sum_i t_i \beta_i)^d$. Using the notation of the lemma, we thus have
\[P_d(\underline \tau) = \bigl(\eta+\textstyle\sum_i \tau_i \beta_i\bigr)^d=\vol\bigl(\eta+\textstyle\sum_i \tau_i \beta_i\bigr) \quad \text{for }|\underline \tau|<c.\]
By Proposition \ref{volpushforward}, $P_d(\underline \tau)$ has a maximum at $\underline \tau =0$, from which we derive that
\[\eta^{d-1} \cdot \beta = 0 \quad \foral \beta \in  \ker \bigl(f_*: \NeSe{Y}_\R \to \NeSe{X}_\R\bigr)\]
and
\[\eta^{d-2} \cdot \beta^2 \leq 0 \quad  \foral \beta \in  \ker f_*,\] the last statement meaning nothing else but that the Hessian of $P_d(\underline \tau)$ at $\underline \tau =0$ is negative semidefinite. Now for any $\beta \in  \ker f_*$ with $\eta^{d-1} \cdot \beta = 0$ and $\eta^{d-2} \cdot \beta^2 = 0$,  \cite[Theorem 9.6.3]{Kle05} implies $\beta = 0$. Therefore the Hessian of $P_d(\underline \tau)$ at $\underline \tau =0$ is negative definite, i.e.\ the assumptions of Lemma \ref{polymax} are fulfilled. Hence there exists an $R>0$ such that
\[\begin{split}
\max \Bigl\{\chi\bigl(Y,m \eta + &\textstyle\sum_{i=1}^s t_i \beta_i\bigr) \ \Bigl| \ \underline t \in m \underline\delta + \mb Z^s, |\underline t|< cm \Bigr\}\\
 &= \max \Bigl\{\chi\bigl(Y,m \eta + \textstyle\sum_{i=1}^s t_i \beta_i\bigr) \ \Bigl| \ \underline t \in m \underline\delta + \mb Z^s, |\underline t|< R \Bigr\}.
\end{split}\]
For fixed $r(m)$, the right-hand side is just the maximum of finitely many fixed polynomials in $m$ (namely the $P(m, \underline t)$ for which $\underline t - r(m)\underline\delta \in \Z^s$ and $|\underline t|<R$), so polynomial itself for $m \gg 0$. Our assertion follows.
\end{proof}

\begin{lemma}\label{polymax}
Let $P(m,\underline t)\in \R[m,\underline t]$ be a polynomial of degree $d$, and let $P_d(m, \underline t)$ denote its highest degree homogeneous part. Assume that $\tilde P_d(\underline \tau)=m^{-d} P_d(m, \underline t)$ (with $\underline \tau = m^{-1} \underline t$) has a local maximum at $\underline \tau=0$, and that the Hessian matrix of $\tilde P_d(\underline \tau)$ at $\underline \tau=0$ is negative definite. Then the following holds:
\begin{enumerate}[(a)]
\item There exist a constant $c > 0$ such that for any fixed $m \gg 0$, $P(m, \underline t)$ is strictly concave in $\underline t$ for $|\underline t|<cm$.
\item For any $c$ as above, there exists an $R > 0$ such that for any $\delta \in \R^s$ and any fixed $m \gg 0$, one has
\[\max \bigl\{ P(m, \underline t) \ \bigl| \ \underline t \in m \delta + \mb Z^s, |\underline t|< cm \bigr\}= \max \bigl\{ P(m, \underline t) \ \bigl| \ \underline t \in m \delta + \mb Z^s, |\underline t|< R \bigr\}.\]
\end{enumerate}
\end{lemma}

\begin{proof}
\begin{enumerate}[(a)]
\item
Set $\tilde P(m, \underline \tau)=m^{-d} P(m, \underline t)= \tilde P_d(\underline \tau) + m^{-1} \tilde P_{d-1}(\underline \tau) + \ldots + m^{-d} \tilde P_0(\underline \tau)$. Then there exists a $c > 0$ such that the Hessian of $\tilde P_d(\underline \tau)$ is negative definite at any $\underline \tau$ with $|\underline \tau| \leq c$, and for any fixed $m \gg 0$, this implies that the Hessian of $\tilde P(m, \underline \tau)$ is negative definite for whenever $|\underline \tau|\leq c$. This implies the assertion.
\item
Let $N$ be the Hessian matrix of $\tilde P_d(\underline \tau)$ at $\underline \tau=0$, and let $B \in \R^s$ such that $\tilde P_{d-1}(\underline \tau) = b_0 + B^T \underline \tau + O(|\underline \tau|^2)$. Set
\[\tilde Q(m, \underline \tau)= \tilde P(m, \underline \tau - m^{-1} N^{-1} B)- \tilde P(m, - m^{-1} N^{-1} B).\]
 Then $\tilde Q(m, \underline \tau) = \frac{1}{2}\underline \tau^T N \underline \tau + O(m^{-2} |\underline \tau|) + O(m^{-1} |\underline \tau|^2) + O(|\underline \tau|^3)$.

Take any $R_1>0$ such that every ball of radius $R_1$ in $\R^s$ contains a point of $\Z^s$, and fix $R_2 > \sqrt\frac{\lambda_\text{min}} {\lambda_\text{max}} R_1$, where $\lambda_\text{min}< \lambda_\text{max}<0$ are the minimal and the maximal eigenvalue of $N$. Then for any fixed $m \gg 0$, we have
\[\min \biggl\{ \tilde Q(m, \underline \tau) \ \biggr| \ |\underline \tau|=\frac{R_1}{m} \biggr\} > \max \biggl\{ \tilde Q(m, \underline \tau) \ \biggr| \ |\underline \tau|=\frac{R_2}{m} \biggr\}.\]
By concavity, it follows that for $m \gg 0$, any value of $\tilde P(m, \underline \tau)$ within the ball $|\underline \tau -\frac{1}{m} N^{-1} B| \leq \frac{R_1}{m}$ is greater than any value of  $\tilde P(m, \underline \tau)$ outside of the ball $|\underline \tau -\frac{1}{m} N^{-1} B| \leq \frac{R_2}{m}$. Taking $R = R_2 + |N^{-1} B|$ and using our assumption on $R_1$, our assertion follows.\qedhere
\end{enumerate}
\end{proof}

\subsection{A revised density}

The only drawback of Theorem \ref{dens1} is that it does not describe the asymptotic decomposition behaviour of all divisors, but only of the ones that are linearly equivalent to multiples of a given ample one. This subsection shows that the result indeed extends to all divisors as long as we use a reasonable method to determine their `size'.

So let us assume that we are in the situation of Theorem \ref{dens1}. We fix an ample class $\eta \in \NeSe{X}_\R$ and use the degree function $\deg_\eta: \NeSe{X}_\R \to \R, \ \delta \mapsto \eta^{d-1} \cdot \delta$, which has been described in Remark \ref{Hvolrem}.  For any $r \in \R$, set  $\mf{P}_{r}=\mf{P}_{r,\eta}$ and $\mf{D}^\Cc_{r}=\mf{D}^\Cc_{r,\eta}$ to be the reduced schemes representing the geometrically integral divisors on $X$ of degree $\leq r$, respectively those that additionally are unramified in the cover $Z \to X$ and have geometric decomposition class $\Cc$.

Then we can state a modified density theorem which describes the asymptotic decomposition behaviour of all Cartier divisors sorted by this notion of degree:

\begin{theo}\label{revdens}
In the above situation, $\mf{P}_{r}$ and $\mf{D}^\Cc_{r}$ exist for any $r\in \R$. Furthermore, we have
\[\lim_{r \to \infty} \frac{\dim  \mf{P}_{r}}{\dim \mf{D}^\Cc_{r}} = \frac{1}{(G:\mc{C})^{d-1}}.\]
\end{theo}

\begin{proof}
Let us first assume that $\eta = \cln{D_0}$ for some ample Cartier divisor $D_0$.

We have
\[\mf{P}_r = \bigcup_{\beta \in \overline\Eff(X)\!,\, \deg(\beta) \leq r} \DIV_{X/k}^\beta \cap \GIDIV_{X/k},\]
so
\[\begin{split}
\dim \mf{P}_r & = \max \bigl\{\dim \mf{P}_{D}\bigl|\deg(D) \leq r\bigr\} + O(1) \\
&\leq \max \bigl\{ h^0(X,D)\bigl|\deg(D) \leq r\bigr\} + O(1).
\end{split}\]
Using Proposition \ref{Hvol} with $S = \{ \beta \in \overline\Eff(X) \ | \ \deg(\beta)=1\}$,  we get \[\limsup_{r \to \infty} \frac{\dim \mf P_r}{r^d/d!} \leq \max \bigl\{\vol(s)\bigl| s \in S \bigr\} = \frac{1}{\vol(D_0)^{d-1}}.\]
On the other hand, since $\deg(m D_0) = m \vol(D_0) \leq r$ for $m \leq \left\lfloor \frac{r}{\vol(D_0)} \right\rfloor$, we have
\[\dim \mf P_r \geq \dim \mf{P}_{\bigl(\left\lfloor \frac{r}{\vol(D_0)} \right\rfloor D_0\bigr)} = \frac{r^d}{d! \vol(D_0)^{d-1}} +O(r^{d-1}). \]
Therefore,
\[\lim_{r \to \infty} \frac{\dim \mf P_r}{r^d/d!}  = \frac{1}{\vol(D_0)^{d-1}}.\]

Similarly, we derive
\[\begin{split}
\dim \mf{D}^\Cc_r & = \max \bigl\{ \dim \mf{D}^\Cc_{D} \ \bigl| \ \deg(D) \leq r \bigr\} + O(1)\\
 &=\max \bigl\{ \dim \mf{S}^{Z/H}_{D} \ \bigl| \ \deg(D) \leq r \bigr\} + O(1)
 \end{split}\] for any representative $H$ of $\Cc$. In the exact same way as in the proof of Lemma \ref{upperbound}, we derive
\[\limsup_{r \to \infty} \frac{\dim \mf D^\Cc_r}{r^d / d!} \leq \frac{1}{\bigl((G:H) \vol(D_0)\bigr)^{d-1}}.\]
On the other hand, the argument in the proof of Lemma \ref{lowerbound} shows that for $r \gg 0$,
\[\dim \mf D^\Cc_r \geq \dim \mf L_{\bigl(\left\lfloor\frac{r}{\vol(f^* D_0)}\right\rfloor f^* D_0\bigr)} = \frac{r^d}{d! \bigl((G:H)\vol(D_0)\bigr)^{d-1}} +O(r^{d-1}),\]
where $f$ is the canonical map $Z/H \to X$.
This implies that $\frac{\dim \mf D^\Cc_r}{r^d/d!}$ converges to $\bigl((G\! : \!H)\vol(D_0)\bigr)^{1-d}$. From the limits for $\dim \mf P_r$ and $\dim \mf D^\Cc_r$, the assertion follows.

To get the assertion for arbitrary ample classes, we first remark that replacing $\eta$ by $\lambda \eta$ with $\lambda \in \R$ does not change the statement above, and then we extend by continuity arguments.
\end{proof}

\subsection{Finite fields}\label{secfin}

We stay in the same context as before, but assume furthermore that $k=\F_q$ is a finite field. Then we can actually count divisors: Fix some ample class $\eta \in \NeSe{X}_\R$, and set $\deg=\deg_\eta$. For $r \in \R$, let $p_\#(r)$ be number of Cartier prime divisors of degree at most $r$, and let $d^\mc{C}_\#(r)$ be the number of those divisors that additionally are unramified in the Galois cover $g:Z\to X$ and have decomposition class $\Cc$.

\begin{theo}\label{dens3}
Under the above assumptions,
\[\lim_{r\to\infty}\frac{\log d^\mc{C}_\#(r)}{\log p_\#(r)} = \frac{1}{(G:\mc{C})^{d-1}}.\]
\end{theo}

Before giving the proof, we need some lemmata:

\begin{lemma}\label{ndenlo}
Let $D_0$ be a very ample divisor on $X$. Then \[\liminf_{m \to \infty}\frac{\# \mf P_{mD_0}(k)}{\# H^0(X, mD_0)}>0.\]
\end{lemma}

\begin{proof}
Let $X \hookrightarrow \mb P^N_k$ be the projective embedding corresponding to $\O_X(D_0)$. Then for $m \geq h^0(X,\O_X)-1$, the map
\[\phi_m : S_m := H^0(\mb P^N_k, \O_{\mb P^N}(m)) \to H^0(X,mD_0)\]
is surjective (\cite[2.1]{Poo04}); thus
\[ \frac{\# \mf P_{mD_0}(k)}{\# H^0(X,mD_0)} = \frac{\# (\mc P \cap S_m)}{\# S_m},\]
where $\mc P$ denotes the set of all $f \in S_{\text{homog}} := \bigcup_{m=0}^\infty S_m$ such that the scheme-theoretic intersection $H_f \cap X$ of the hypersurface $H_f$ of $\mb P^N_k$ defined by $f$ with $X$ is geometrically integral.

For $f \in S_{\text{homog}}$ to be in $\mc P$, it is sufficient that $X \cap H_f$ is normal. For then, $X \cap H_f$ is geometrically normal (\cite[Proposition 6.7.7]{Gro65}) and geometrically connected by Grothendieck's connectedness theorem (\cite[Theorem III.7.9]{Har77}), so geometrically integral.

Using Serre's criterion (\cite[Theorem 8.2.23]{Liu02}), it is enough for us to check whether $X \cap H_f$ fulfills $R_1$ and $S_2$. This will be done in several steps. In the following, let $X^{\text{reg}}$ and $X^{\text{sing}}$ denote the regular and the singular locus of $X$, respectively (since $X$ is normal, $X^{\text{reg}}$ is smooth and $\codim(X^{\text{sing}}, X)\geq 2$).

In order to consider Serre's condition $S_2$, define
\[Z_r(Y)= \{y \in Y | \codepth \O_{Y,y} := \dim \O_{Y,y} - \depth \O_{Y,y} > r\}\]
for any Noetherian scheme $Y$. Then $Y$ fulfills $S_2$ if and only if
\[\codim(Z_r(Y),Y) > r+2 \text{ for all } r \geq 0 \quad \text{ (\cite[Proposition 5.7.4]{Gro65}).}\]
Since $X$ is normal, we have $\codim(Z_r(X),X) > r+2 \foral r \geq 0$. On the other hand, if $x \in X$ lies in $H_f$ $(f \neq 0)$, then $f$ maps to a regular element of $\f{m}_x$, so \[\codepth \O_{X,x} = \codepth \O_{X \cap H_f, x}\] by \cite[Proposition 0.16.4.10]{Gro64}; thus $Z_r(X\cap H_f)= Z_r(X) \cap H_f$. So in order for $f$ to fulfill $\codim(Z_r(X \cap H_f),X \cap  H_f) > r+2 \foral r \geq 0$ and hence $S_2$, it will be sufficient that $H_f$ intersects all irreducible components of all $Z_r(X)$ properly, i.e.\ $H_f$ does not contain any of the irreducible components of any $Z_r(X)$. There are only finitely such irreducible components, since $Z_r(X)$ is empty for $r \geq d$.

Set $Z$ to be the finite reduced subscheme of $X$ consisting of all closed points which are an irreducible component of either one of the $Z_r(X)$ or of $X^{\text{sing}}$, and set $\mc Q$ to be the set of all $f \in S_{\text{homog}}$ such that $H_f$ contains at least one of the positive dimensional irreducible components of either one of the $Z_r(X)$ or of $X^{\text{sing}}$.

We claim that
\begin{equation}\label{eqpoo1}
\frac{\# (\mc Q \cap S_m)}{\# S_m} \xrightarrow{m \to \infty} 0.
\end{equation}
This follows from the fact that given any irreducible subvariety $W$ of $\P^N_k$ of dimension $\geq 1$ and $\tilde {\mc Q} := \{f \in S_{\text{homog}} \ | \ W \subset H_f\}$, we have
\[\frac{\# (\tilde{\mc Q} \cap S_m)}{\# S_m} =  \frac{1} {\# H^0(W,O_W(m))} \xrightarrow{m \to \infty} 0\]
by Riemann-Roch.

We now turn towards the $R_1$ property. If $f \notin \mc Q$ and $H_f \cap Z = \emptyset$, then $\codim (X^{\text{sing}}\cap H_f, X \cap H_f) \geq 2$; thus $X \cap H_f$ fulfills $R_1$ if  $X^{\text{reg}}\cap H_f$ does. So it is sufficient that $X^{\text{reg}} \cap H_f$ is smooth. If $T \subset H^0(Z, \O_Z)$ is the (nonempty) set of all sections which do not vanish at any point of $Z$, we set
\[\mc P' := \{ f \in S_{\text{homog}}: H_f \cap X^{\text{reg}} \text{ is smooth, and } f|_Z \in T\}.\]
Using Poonen's Bertini theorem for finite fields (\cite[1.2]{Poo04}), we get
\begin{equation}\label{eqpoo2}
\lim_{m \to \infty} \frac{\# (\mc P' \cap S_m)}{\# S_m} = \frac{\# T}{\# H^0(Z, \O_Z)} \zeta_{X^\text{reg}}(d+1)^{-1} =: C > 0,
\end{equation}
where $\zeta_{X^\text{reg}}$ is the zeta function of $X^\text{reg}$.

Putting all these pieces of information together, we have \[\mc P \supset \mc P' - \mc Q;\]
therefore \eqref{eqpoo1} and \eqref{eqpoo2} imply
\[ \liminf_{m \to \infty} \frac{\# \mf P_{mD_0}(k)}{\# H^0(X,mD_0)} = \liminf_{m \to \infty}\frac{\# (\mc P \cap S_m)}{\# S_m} \geq C > 0.\qedhere\]
\end{proof}

\begin{lemma}\label{ndenup}
Let $H$ be a subgroup of $G$, and let $\eta' \in \NeSe{Z}^H_\R$ be an ample class. Then
\[\limsup_{r \to \infty}\frac{\log_q \# \left(\bigcup_{\beta \in \NeSe{Z}^H; \ \deg_{\eta'}(\beta) \leq r}\left(\DIV^\beta_{Z/k}\right)^H\right)(k)}{r^d / d!} \leq \frac{1}{\# H \cdot \vol(\eta')^{d-1}}.\]
\end{lemma}

\begin{proof}
Let us first give a bound on the number of $\DIV_{Z/k}^\beta$ in this inequality: Let $\mu$ be the Haar measure on $\NeSe{Z}^H_\R$ such that $\mu(\NeSe{Z}^H_\R/\NeSe{Z}^H)=1$. Then the set $\{\beta \in \overline\Eff(Z) \cap \NeSe{Z}^H_\R \ | \ \deg_{\eta'}(\beta) \leq 1\}$ is compact and convex, hence has finite measure $V$; by standard combinatorial arguments, one can deduce that for any $r \in \R_{\geq 0}$, we have
\begin{equation}\label{ndens1}
\# \bigl\{ \beta \in \overline\Eff(Z)\! \cap\! \NeSe{Z}^H \ \bigl| \ \deg_{\eta'}(\beta) \leq r \bigr\} =  V r^l + O(r^{l-1}),
\end{equation}
where $l$ is the dimension of the $\R$-vector space $\NeSe{Z}^H_\R$.

On the other hand, a $k$-rational point of $\left(\DIV_{Z/k}^{\beta}\right)^H$ maps to a $k$-rational point of $\left(\PIC_{Z/k}^{\beta}\right)^H$ under the Abel map; considering the fibers of this map, we get
\[
\# \left(\DIV_{Z/k}^{\beta}\right)^{H}\!(k) \leq \# \PIC_{Z/k}^{\tau} (k) \cdot  \max \Bigl\{ \# H^0(Z,\Lc)^H \ \Bigl| \ \Lc \in \Pic(Z)^H, \cln{\Lc}=\beta \Bigr\}
\]
using the same arguments as for the proof of \eqref{equp2}. Combining this with \eqref{ndens1}, we get
\begin{equation}\label{ndens3}\begin{split}
\log_q \Biggl(& \sum_{\beta \in \NeSe{Z}^H\!, \,\deg_{\eta'}(\beta) \leq r} \# \left(\DIV^\beta_{Z/k}\right)^H(k)\Biggr)\\&\leq C + l \log_q r + \max \Bigl\{ \dim H^0(Z,\Lc)^H \ \Bigl| \ \Lc \in \Pic(Z)^H, \deg_{\eta'}(\Lc) \leq r\Bigr\}
\end{split}\end{equation}
for some positive constant $C$. Using Proposition \ref{Hvol} in the way described in Remark \ref{Hvolrem}, this gives the assertion.
\end{proof}

\begin{proof}[Proof of Theorem \ref{dens3}]
It suffices to prove the theorem in the case when $\eta = \cln{D_0}$ for some very ample divisor $D_0$; by taking rational multiples, we then deduce the statement for any ample $\eta \in \NeSe{X}_\Q$ and finally for any ample $\eta \in \NeSe{X}_\R$ by continuity arguments.

As in the proof of Theorem \ref{dens1}, the theorem follows from the following result:

\noindent\emph{Claim:} Let $H$ be a subgroup of $G$, $Y=Z/G$. For any Cartier divisor $D$ on $X$, let $s^Y_\#(r)$ be number of Cartier prime divisors of degree at most $r$ that split in $Y$. Then
\[\lim_{r\to\infty}\frac{\log s^Y_\#(r)}{\log p_\#(r)} = \frac{1}{(G:H)^{d-1}}.\]

Let us first deal with the asymptotic behaviour of the denominator: By Lemma \ref{ndenup} (with $Z=X$, $H=1$ and $\eta'=\eta$), we have
\begin{equation}\label{eqnum1}
\limsup_{r \to \infty} \frac{\log_q p_\#(r)}{r^d/d!} \leq \frac{1}{\vol(D_0)^{d-1}}.
\end{equation}
On the other hand, $p_\#(r)$ is bounded below by $\# \mf P_{\left(\left\lfloor \frac{r}{\vol(D_0)}\right\rfloor D_0\right)}(k)$. By Lemma \ref{ndenlo}, we have
\begin{equation}\label{limpnum}
\log_q \# \mf P_{mD_0}(k) \sim \log_q \# H^0(X, mD_0) = h^0(X,mD_0) \sim \frac{\vol(D_0)}{d!} m^d,
\end{equation}
where $a_m \sim b_m$ means $a_m/b_m \xrightarrow{m\to\infty} 1$. Therefore,
\begin{equation}\label{eqnum2}
\liminf_{r \to \infty} \frac{\log_q p_\#(r)}{r^d/d!} \geq \liminf_{r \to \infty} \frac{\vol(D_0)\left\lfloor \frac{r}{\vol(D_0)}\right\rfloor^d}{r^d} = \frac{1}{\vol(D_0)^{d-1}}.
\end{equation}

We turn our attention to the asymptotic behaviour of the numerator. In a similar way as in the proof of \eqref{equp1}, we deduce
\[s^Y_\#(r) \leq \sum_{\beta \in \NeSe{Z}^H\!,\, \deg_{\eta'}(\beta) \leq \# H \cdot r + C} \# \left(\DIV^\beta_{Z/k}\right)^H(k).
\]
where $C$ is some constant and $\eta' = \cln{g^* D_0}$. Lemma \ref{ndenup} then implies that
\begin{equation}\label{eqnum3}
\limsup_{r \to \infty} \frac{\log_q s^Y_\#(r)}{r^d/d!} \leq \frac {(\# H)^{d-1}}{\vol(g^* D_0)^{d-1}}= \frac{1}{\bigl((G:H)\vol(D_0)\bigr)^{d-1}}.
\end{equation}

The lower bound follows similarly to the proof of Lemma \ref{lowerbound}: Let $v_\#(m)$ denote the number of prime divisors $E'$ in the complete linear system $|m f^* D_0|$ for which $h^* E'$ is also a prime divisor. From \eqref{limpnum} and \eqref{eqnum3}, we get
\[\log_q v_\#(m) \sim \log_q \# H^0(Y,m f^* D_0) = h^0(Y, m f^* D_0) \sim (G:H)\vol(D_0)\frac{m^d}{d!},\]
where again $a_m \sim b_m$ means $a_m/b_m \xrightarrow{m\to\infty} 1$.

If such a prime divisor $E'$ does not map to a prime divisor in $|mD_0|$ under the push-forward map $f_*$, we know by the proof of Lemma \ref{lowerbound} that $h^* E'$ has to lie in one of the spaces $|m g^* D_0|^{H'}$ for some $H \lneq H' \leq G$. The number of elements in these spaces can be bounded by Lemma \ref{ndenup} (or Proposition \ref{Hvol}) to show that asymptotically, the push-forward of almost every Cartier divisor counted by $v_\#(m)$ is a prime divisor in $|(G:H)m D_0|$. In particular,
\begin{equation}\label{eqnum4}
\begin{split}
\liminf_{r \to \infty} \frac{\log_q s^Y_\#(r)}{r^d/d!}
& \geq \liminf_{r \to \infty} \frac{\log_q \left(\frac{1}{(G:H)} v_\#\Bigl( \left\lfloor \frac{r}{(G:H) \vol(D_0)}\right\rfloor\Bigr)\right)}{r^d/d!} \\
& = \frac{1}{\bigl((G:H)\vol(D_0)\bigr)^{d-1}}.
\end{split}
\end{equation}
Combining \eqref{eqnum1}, \eqref{eqnum2}, \eqref{eqnum3} and \eqref{eqnum4}, our theorem follows.
\end{proof}

\subsection{The Theorem of Bauer-Schmidt}

\begin{theo}\label{bauer}
Let $X, Y, Z$ be normal geometrically integral quasiprojective varieties of dimension $d \geq 2$ over a field $k$ of characteristic zero; let $f: Y \to X$ be a finite branched cover, $g : Z \to X$ a finite branched Galois cover. Fix an integer $r$ with $0 < r < d$. Then the following are equivalent:
\begin{enumerate}[(a)]
\item $f: Y \to X$ factors through $g: Z \to X$.
\item Every point $x \in X$ of codimension $r$ that splits in $Y$ splits in $Z$.
\end{enumerate}
\end{theo}

\begin{proof}
The implication $(a) \Ra (b)$ is immediate. For the converse, embed $K(Y)$ and $K(Z)$ into some algebraic closure $\overline{K(X)}$ of $K(X)$, and let $L$ be the smallest Galois extension of $K(X)$ inside $\overline{K(X)}$ containing both $K(Y)$ and $K(Z)$. Let $W$ be the normalization of $X$ in $L$. Then $W$ is a normal geometrically integral quasiprojective variety of dimension $d$ over $k$, and if we set $G= \Gal(L|K(X))$, $H= \Gal(L|K(Y))$, $N = \Gal(L|K(Z))$, then $W/G \cong X$, $W/H \cong Y$ and $W/N \cong Z$ because of the normality conditions. We have to show that if $f: Y \to X$ does not factor through $g: Z \to X$ (or, equivalently, $H \not\subseteq N$), there is a point of codimension $r$ in $X$ that splits in $Y$ but does not split in $Z$.

Applying Proposition \ref{exptsdec} below to the Galois cover $W \to X$ and the conjugacy class $\Cc$ of $H$ inside $G$, we get infinitely many points of codimension $r$ in $X$ that are unramified in $W$ and have $\Cc$ as their decomposition class. By construction, every such point splits in $Y$; if any such point split in $Z$, then some conjugate of $H$ would have to be a subgroup of $N$. As $N$ is normal in $G$, this implies $H \subseteq N$, contradiction.
\end{proof}

\begin{prop}\label{exptsdec}
Let $f: Z \to X$ be a finite branched Galois cover (with Galois group $G$) of geometrically integral quasiprojective varieties of dimension $d \geq 2$ over a field $k$ of characteristic zero. Then for any positive integer $r < d$ and for any conjugacy class $\Cc$ of subgroups of $G$, there are infinitely many points of codimension $r$ in $X$ that are unramified in the cover $Z \to X$ and have decomposition class $\Cc$.
\end{prop}

\begin{proof}
The proof will be done by induction on $r$.

In the case $r=1$, by completing and normalizing we may assume that $Z$ and $X$ are normal and projective; since this adds only finitely many points of codimension one, our claim remains unchanged. Under these assumptions, the claim follows from Remark \ref{remratpts}.

Now assume $r>1$. By Theorem \ref{dens1}, we know there exists a geometrically integral divisor $D$ on $X$ which is unramified in $Z$ and whose decomposition class is the (conjugacy class of) the full group $G$, in other words, $D$ stays prime in $Z$. Take $X_1$ to be the closed subscheme of $X$ corresponding to $D$ and let $Z_1$ be the preimage in $Z$. Then $Z_1 \to X_1$ is a Galois cover of geometrically integral varieties over $k$ of dimension $d-1 > r-1 \geq 1$. By induction hypothesis, there are infinitely many points of codimension $r-1$ in $X_1$ which have decomposition class $\Cc$ in $Z_1$. Since all these points are points of codimension $r$ in $X$ that have decomposition class $\Cc$ in $Z$, we are done.
\end{proof}

\begin{cor}\label{corbauer}
Let $f: Y \to X$ be a finite branched cover of normal geometrically integral quasiprojective varieties of dimension $d \geq 2$ over a field $k$ of characteristic zero. Fix an integer $r$ with $0 < r < d$. Then the following are equivalent:
\begin{enumerate}[(a)]
\item $f$ is a Galois cover.
\item Every point $x \in X$ of codimension $r$ that splits in $Y$ splits completely, i.e.\ (it is unramified in $Y$ and) for every point $y \in Y$ with $f(y)=x$, we have $\kappa(y)=\kappa(x)$.
\end{enumerate}
\end{cor}

\begin{proof}
The implication $(a) \Ra (b)$ is immediate. For the converse, let $g: Z \to X$ be the Galois closure of $Y \to X$ (see Remark \ref{makeGal}), and let  $G$ and $H$ be the Galois groups of $g: Z \to X$ and $h:Z \to Y$, respectively. We have to show $Z = Y$ (or equivalently $G=H$). The proof follows from Theorem \ref{bauer} and the following simple and well-known fact from Hilbert's decomposition theory: A point $x \in X$ splits completely in $Y$ if and only if it splits (completely) in $Z$ (the argument can be found e.g.\ in \cite[proof of Corollary VI.3.8]{Neu99}).
\end{proof}

\begin{rem}\label{remschmidt}
The statements of \ref{bauer} -- \ref{corbauer} hold more generally for finite branched covers of normal quasiprojective varieties of dimension $d \geq 2$ over a field $k$ of arbitrary characteristic. This can be deduced from \cite{Sch34}: Using the fact that the function field $k(t)$ is hilbertian, the proof of Theorem C of \emph{loc.cit.}\ implies the case $r=1$ of Proposition \ref{exptsdec}. The rest of the statements then follows in the same way as above.
\end{rem}

\addcontentsline{TOC}{section}{Bibliography}

\bibliography{references}
\bibliographystyle{alpha}

\end{document}